\documentclass[11pt]{article}

\usepackage{a4wide,epsfig}
\usepackage[a4paper, top=3cm]{geometry} 
\usepackage{amsfonts,amsmath}
\usepackage{graphicx}
\usepackage{amssymb}
\usepackage{hyperref}
\usepackage{url}
\usepackage{booktabs}
\usepackage{amsfonts}
\usepackage{nicefrac}
\usepackage{microtype}
\usepackage{pifont}
\usepackage{graphicx}
\usepackage{amsmath}
\usepackage{doi}
\usepackage{amsthm}
\usepackage{dsfont}
\usepackage{mathtools}
\usepackage{tcolorbox}
\usepackage{amsmath}
\usepackage{physics}
\newcommand{\eremk}{\hbox{}\hfill\rule{0.8ex}{0.8ex}}
\usepackage{bbm}
\usepackage{wasysym}
\usepackage{caption}
\usepackage{subcaption}
\usepackage{enumerate}
\usepackage{enumitem}

\definecolor{myblue}{HTML}{1f77b4}
\definecolor{myorange}{HTML}{ff4c0e}
\definecolor{myyellow}{HTML}{bcbd22}
\definecolor{mypurple}{HTML}{9467bd}
\definecolor{mygreen}{HTML}{2ca02c}

\newtheorem{theorem}{Theorem}[section]
\newtheorem{assumption}[theorem]{Assumption}
\newtheorem{lemma}[theorem]{Lemma}
\newtheorem{proposition}[theorem]{Proposition}
\newtheorem{corollary}[theorem]{Corollary}
\newtheorem{remark}[theorem]{Remark}

\DeclareFontFamily{U}{matha}{\hyphenchar\font45}
\DeclareFontShape{U}{matha}{m}{n}{
<-6> matha5 <6-7> matha6 <7-8> matha7
<8-9> matha8 <9-10> matha9
<10-12> matha10 <12-> matha12
}{}
\DeclareSymbolFont{matha}{U}{matha}{m}{n}
\DeclareFontFamily{U}{mathx}{\hyphenchar\font45}
\DeclareFontShape{U}{mathx}{m}{n}{
<-6> mathx5 <6-7> mathx6 <7-8> mathx7
<8-9> mathx8 <9-10> mathx9
<10-12> mathx10 <12-> mathx12
}{}
\DeclareSymbolFont{mathx}{U}{mathx}{m}{n}
\DeclareMathDelimiter{\vvvert} {0}{matha}{"7E}{mathx}{"17}%

\DeclarePairedDelimiterX{\normiii}[1]
{\vvvert}
{\vvvert}
{\ifblank{#1}{\:\cdot\:}{#1}}

\renewcommand{\phi}{\varphi}

\DeclareMathOperator{\Id}{Id}

\newcommand{\bx}{\boldsymbol{x}}

\newcommand{\R}{\mathbb{R}}

\newcommand{\N}{\mathbb{N}}

\newcommand{\A}{\mathcal{A}}

\newcommand{\bh}{\boldsymbol{h}}

\newcommand{\Pt}{\Pi_{h_t}^{\partial_t}}
\newcommand{\Pg}{\Pi^{\nabla}_{h_{\bx}}}
\newcommand{\Px}{\Pi^0_{h_{\bx}}}

\newcommand{\NOmega}{\mathcal{N}^\Omega}
\newcommand{\NOmegae}{\mathcal{N}^\Omega_e}
\newcommand{\NOmegahx}{\mathcal{N}^\Omega_{h_{\bx}}}
\newcommand{\NOmegaeh}{\mathcal{N}^\Omega_{e,\bh}}

\title{Inf-sup stable space--time discretization of the wave equation based on a first-order-in-time variational formulation}
\author{$^a$Matteo~Ferrari, $^a$Ilaria~Perugia, $^a$Enrico~Zampa \vspace{0.5cm}}

\date{
$^a$Faculty of Mathematics, University of Vienna\\
Oskar-Morgenstern-Platz 1, 1090 Vienna, Austria
}

\begin{document}
\maketitle

\begin{abstract}
\noindent
In this paper, we present a conforming space--time discretization of the wave equation based on a first-order-in-time variational formulation with exponential weights in time. We analyze the method, showing its stability without imposing any restrictions on the mesh size or time step, and proving quasi-optimal convergence for any choice of space--time tensor product discrete spaces that satisfies standard approximation assumptions. Numerical examples are provided to support the theoretical findings.
\end{abstract}


\section{Introduction}

The need for accurate wave simulation has led to major advances in theory and numerical methods. In recent years, an increasing number of \emph{conforming} space–time methods have been proposed and are actively being developed. In this context, two key properties properties are typically sought: unconditional stability and quasi-optimal convergence. Here is a brief overview of these methods.\medskip

\begin{itemize}[leftmargin=2em]
\item[1)] {\bf Second-order-in-time methods}, in which the wave field is the sole unknown:
\begin{itemize}[leftmargin=2em]
    \item[1.a)] with double integration by parts in space and time (ultra-weak formulation)
\begin{itemize}[leftmargin=1.5em]
    \item[$\circ$] \cite{HenningPalittaSimonciniUrban2022}: Petrov–Galerkin methods with trial spaces selected to ensure optimal inf-sup stability, following a DPG approach \cite{DemkowiczGopalakrishnanNagarajSepulveda2017}; $H^2$-regularity in both space and time is required for the test functions;
    \end{itemize}
    \item[1.b)] with single integration by parts in space and time (weak formulation)
    \begin{itemize}[leftmargin=1.5em]
    \item[$\circ$] \cite{SteinbachZank2019, Zank2021}: Galerkin methods employing continuous piecewise polynomials in time, stabilized by inserting a projection in time into discontinuous polynomials of one degree lower within the grad-grad term;
    \item[$\circ$] \cite{FraschiniLoliMoiolaSangalli2023, FerrariFraschini2024}: Galerkin methods employing maximal regularity splines in time, stabilized through the addition of a non-consistent penalty term;
    \item[$\circ$] \cite{BignardiMoiola2023}: coercive Galerkin methods, in which Morawetz multipliers are employed as test functions; $H^2$-regularity in both space and time is required for the trial and test functions;
    \item[$\circ$] \cite{LoscherSteinbachZank2023, FerrariLoscherZank2025}: a Galerkin method with continuous piecewise linear functions in time, modified by a transformation operator known as modified Hilbert transform (see \cite{SteinbachZank2020}) to guarantee stability;
    \item[$\circ$] 
\cite{KotheLoscherSteinbach2023}: a least-squares method with continuous piecewise linear functions on unstructured meshes;
    \end{itemize}
    \item[1.c)] with integration by parts in space but not in time (weak formulation in space)
    \begin{itemize}[leftmargin=1.5em]
    \item[$\circ$] \cite{French1993}: time-stepping-like schemes (discontinuous test functions in time) employing continuous piecewise polynomials for the trial functions and exponential weights in time;
    \item[$\circ$] \cite{Walkington2014, DongMascottoWang2024}: time-stepping-like schemes employing continuous piecewise polynomials for the trial functions;
    \item[$\circ$] \cite{FerrariPerugia2025}: Galerkin methods with splines of arbitrary degree and exponential weights in time, exhibiting quasi-optimal convergence when the difference between polynomial degree and spline regularity is odd.
    \end{itemize}
\end{itemize} 
\item[2)] {\bf First-order-in-time methods}, in which the unknowns are the wave field and its velocity: 
\begin{itemize}[leftmargin=2em]
\item[$\circ$] \cite{BalesLasiecka1994, FrenchPeterson1996, Gomez2025}: time-stepping-like schemes employing continuous piecewise polynomials in time for the trial functions;
\item[$\circ$] \cite{FerrariFraschiniLoliPerugia2024}: Galerkin methods employing maximal regularity splines in time for the trial functions, and their derivative space for the test functions;
\item[$\circ$] \cite{FuhrerGonzalezKarkulik2025}: first-order least-squares methods in both time and space;
\item[$\circ$] \cite{AnselmannBause2020,AnselmannBauseBecherMatthies2020}: Petrov-Galerkin collocation methods in time employing arbitrary regularity splines as trial functions.
\end{itemize}
\end{itemize}
\medskip

\noindent
In several of the above references, discrete space--time tensor product spaces are employed. As this brief overview shows, regularity assumptions in space or time are often required, and unconditional stability (or convergence) is often proven for a particular type of time-discrete spaces (e.g. polynomials or maximal regularity splines).
\vspace{-3pt}

\noindent
In this work, we consider conforming first-order-in-time space--time methods with exponential weights in time, where the discrete tensor product spaces are kept general, and both stability and convergence are guaranteed under standard approximation assumptions. The core idea is to exploit the exponential weights similarly to~\cite{Walkington2014}, and establish an inf-sup condition for the continuous variational formulation in suitable norms, 
under the assumption that the test functions are time derivatives of the trial functions. A slightly stronger solution regularity than the minimal one is required, as, e.g., in~\cite{BignardiMoiola2023, FuhrerGonzalezKarkulik2025, FerrariPerugia2025}. While in~\cite{Walkington2014} the exponential weights are used only in the analysis, we keep them in the definition of the numerical method, as in~\cite{French1993} and then in~\cite{FerrariPerugia2025}.
The stability analysis of the \emph{discrete} formulation is a straightforward extension of that of the continuous formulation, based on the same assumption that the test functions are time derivatives of the trial functions, allowing flexibility in the choice of the (conforming) trial space.

\noindent
The analysis we propose shares several ideas with those in \cite{Gomez2025} and \cite{FerrariPerugia2025}. The main difference with respect to~\cite{Gomez2025}, which considers the same first-order-in-time formulation but without exponential weights, lies in the stability at the continuous level and in the greater generality of our framework, which allows for the use of any discrete space satisfying standard approximation assumptions. In contrast, the difference with respect to the method in~\cite{FerrariPerugia2025}, which also employs exponential weights but within a second-order-in-time formulation, lies in the derivation of quasi-optimal error estimates for more general choices of the discrete spaces in time. While the method in \cite{FerrariPerugia2025} is coercive, quasi-optimality in time there is guaranteed only under specific conditions. As a final note, we emphasize that our inf-sup analysis makes use of a Newton potential and discrete Newton potential operators, similarly to the inf-sup analysis of space--time methods for parabolic problems developed in~\cite{Steinbach2015}. Although the problem considered here is of hyperbolic nature, the introduction of exponential weights allows us to exploit common features with the parabolic case studied in~\cite{Steinbach2015}.
\medskip

\noindent
The paper is organized as follows: in Section~\ref{sec:2}, we introduce the model problem and the regularity assumptions that are used throughout the paper. In Section~\ref{sec:4}, we define and analyze the exponentially-weighted variational formulation of the space--time problem, proving its inf-sup stability and well-posedness.
In Section~\ref{sec:5}, we study  space--time tensor product conforming discretizations, proving unconditional stability and quasi-optimal error bounds in both the energy and $L^2$ norms, under standard approximation assumptions on the discrete spaces. Numerical experiments are presented in Section~\ref{sec:5} to validate the theoretical findings.

\section{Model problem} \label{sec:2}

Let $\Omega$ be a Lipschitz, bounded polytopal domain in $\R^d$ ($d = 1,2,3$), $T$ a positive final time, and~$Q_T := \Omega \times (0,T)$ the space--time cylinder.

\noindent
Given initial data~$U_0 \in H_0^1(\Omega)$, $V_0 \in L^2(\Omega)$ and a forcing term~$F \in L^2(Q_T)$, consider the following problem: find $U :\overline{Q}_T \to \mathbb{R}$ such that
\begin{equation} \label{eq:1}
\begin{cases}
    \partial_t^2 U(\bx,t) - \nabla_{\bx}\cdot(c^2(\bx)\nabla_{\bx} U(\bx,t)) = F(\bx,t) & (\bx,t) \in Q_T, \\
    U(\bx,t) = 0 & (\bx,t) \in  \partial\Omega \times (0,T], \\
    U(\bx,0) = U_0(\bx), \quad \partial_t U(\bx,0) = V_0(\bx) & \bx \in \Omega.
\end{cases}
\end{equation}
Here, $c = c(\bx)$ is the wave velocity, for which we assume, for a constant~$c_0$,
\begin{equation}\label{eq:2}
    c \in L^\infty(\Omega) \quad \text{and} \quad c(\bx) \ge c_0 > 0 \quad \text{for almost every~} \bx \in \Omega.
\end{equation}
A well-posed variational formulation of problem~\eqref{eq:1} reads: find 
\begin{equation} \label{eq:3}
    U \in L^2(0,T; H_0^1(\Omega)) \cap H^1(0,T; L^2(\Omega))
\end{equation}
such that $U_{|_{t=0}} = U_0$ in $H^1_0(\Omega)$, $\partial_t U_{|_{t=0}} = V_0$ in $L^2(\Omega)$ and, for almost every $t \in (0,T)$,
\begin{equation} \label{eq:4}
    \langle \partial_t^2 U(\cdot,t), W \rangle_{H^1_0(\Omega)} + (c^2\nabla_{\bx} U(\cdot,t), \nabla_{\bx} W)_{L^2(\Omega)} = (F(\cdot,t), W )_{L^2(\Omega)}
\end{equation}
for all $W \in H_0^1(\Omega)$, where $\langle\cdot,\cdot\rangle_{H^1_0(\Omega)}$ is the duality product between~$[H_0^1(\Omega)]'$ and~$H^1_0(\Omega)$; see, e.g.,~\cite[Chapter 3. \S 8 Pag. 265]{LionsMagenes1972}.

\noindent
Introducing the new unknown $V := \partial_t U$, we rewrite problem~\eqref{eq:1} as the first-order-in-time system
\begin{equation} \label{eq:5}
\begin{cases}
    \partial_t V(\bx,t) - \nabla_{\bx}\cdot(c^2(\bx)\nabla_{\bx} U(\bx,t)) = F(\bx,t) & (\bx,t) \in Q_T,
    \\ \partial_t U(\bx,t) - V(\bx,t) = 0 & (\bx,t) \in Q_T, 
    \\ U(\bx,t) = 0 & (\bx,t) \in  \partial\Omega \times (0,T),
    \\ U(\bx,0) = U_0(\bx), \quad V(\bx,0) = V_0(\bx) & \bx \in \Omega.
\end{cases}
\end{equation}
If $U$ satisfies the regularity in~\eqref{eq:3}, then
\begin{equation*}
    V \in L^2(Q_T) \cap H^1(0,T; [H_0^1(\Omega)]').
\end{equation*}
For the purposes of the scheme we propose, we restrict our setting to a more regular framework and make the following assumption.
\begin{assumption} \label{ass:21}
Let the wave velocity $c$ be as in~\eqref{eq:2}, $U_0 \in H_0^1(\Omega)$, and
\begin{equation*}
    F \in H^1(0,T;L^2(\Omega)), \quad V_0 \in H^1_0(\Omega), \quad F(\cdot,0) \in L^2(\Omega), \quad \nabla_{\bx} \cdot (c^2 \nabla_{\bx} U_0) \in L^2(\Omega). 
\end{equation*}
\end{assumption}
\noindent
Under Assumption~\ref{ass:21}, one can establish improved regularity for the unique solution to~\eqref{eq:4}; see, e.g.,~\cite[Prop.~4.4]{AbdulleHenning2017}, \cite[Lemma~2.7]{FuhrerGonzalezKarkulik2025}, and \cite[Chapter~7.2]{EvansBook}.
\begin{proposition}\label{prop:22}
Under Assumption~\ref{ass:21}, there exists a unique solution to~\eqref{eq:4} and it satisfies
\begin{equation*}
    U \in L^2(0,T; H_0^1(\Omega)), \quad 
    \partial_t U \in L^\infty(0,T; H_0^1(\Omega)), \quad 
    \partial_t^2 U \in L^\infty(0,T; L^2(\Omega)).
\end{equation*}
Setting \( V = \partial_t U \), this in particular implies
\begin{equation*}
    U \in H^1(0,T; H_0^1(\Omega)), \quad 
    V \in L^2(0,T; H_0^1(\Omega)) \cap H^1(0,T; L^2(\Omega)).
\end{equation*}
\end{proposition}

\section{Space--time variational formulation}\label{sec:4}

In this section, we introduce and analyze a variational formulation of~\eqref{eq:5}. We establish an inf-sup condition in a way that extends straightforwardly to the discrete setting. 
\medskip

\noindent
We set $H_{0,\bullet}^1(0,T) := \{u \in H^1(0,T) : u(0)=0\},$ and consider the following variational formulation: find $U \in H^1_{0,\bullet}(0,T;H^1_0(\Omega))$ and $V \in H^1_{0,\bullet}(0,T;L^2(\Omega))$ such that
\begin{equation} \label{eq:22}
\begin{cases}
    (\partial_t V, \lambda)_{L^2(Q_T)} + (c^2 \nabla_{\bx} U, \nabla_{\bx} \lambda)_{L^2(Q_T)} = (F, \lambda )_{L^2(Q_T)} - (c^2\nabla_{\bx} U_0, \nabla_{\bx} \lambda)_{L^2(Q_T)}, 
    \\ (\partial_t U, \chi)_{L^2(Q_T)} - (V, \chi)_{L^2(Q_T)} = (V_0,\chi)_{L^2(Q_T)},
\end{cases}    
\end{equation}
for all $\lambda \in L^2(0,T;H_0^1(\Omega))$ and $\chi \in L^2(Q_T)$. If $(U,V)$ is a solution to~\eqref{eq:5}, then $(U - U_0,\, V - V_0)$ satisfies~\eqref{eq:22}.

\noindent
We now introduce a modified variational formulation that incorporates weighted scalar products in time, and we study the properties of associated bilinear form.

\subsection{Weighted space--time bilinear form and its properties}

For functions depending solely on time, we define the weighted scalar product in $L^2(0,T)$ and associated norm
\begin{equation} \label{eq:8}
    (u,v)_{L^2_e(0,T)} := \int_0^T u(s) v(s) e^{-s/T} \dd s, \quad\quad\quad \| u \|_{L^2_e(0,T)}^2 := (u,u)_{L^2_e(0,T)}.
\end{equation}
The norm defined in \eqref{eq:8} is equivalent to the classical $L^2$ norm.
The important property of this scalar product is that, for all~$u, v \in H^1(0,T)$,
\begin{equation*}
    (u, \partial_t v)_{L^2_e(0,T)} = -(\partial_t u, v)_{L^2_e(0,T)} + \frac{1}{T} (u,v)_{L^2_e(0,T)} + \frac{1}{e} u(T) v(T)- u(0) v(0),
\end{equation*}
from which it follows that
\begin{equation} \label{eq:9}
    (w, \partial_t w)_{L^2_e(0,T)}  = \frac{1}{2T} \| w \|_{L^2_e(0,T)}^2 + \frac{1}{2e} |w(T)|^2 \qquad \text{for all~} w \in H_{0,\bullet}^1(0,T). 
\end{equation}
Then, for space--time functions, we denote by~$(\cdot,\cdot)_{L^2_e(Q_T)} : L^2(Q_T) \times L^2(Q_T) \to \R$ the weighted scalar product
\begin{equation*}
    ( \chi, \lambda )_{L^2_e(Q_T)} := \int_0^T ( \chi(\cdot,s), \lambda(\cdot,s) )_{L^2(\Omega)} \, e^{-s/T} \, \dd s,
\end{equation*}
and by  $\langle \cdot, \cdot \rangle_e : L^2(0,T;[H^1_0(\Omega)]') \times L^2(0,T;H^1_0(\Omega)) \to \R$ the duality pairing
\begin{equation*}
    \langle \chi, \lambda \rangle_e := \int_0^T \langle \chi(\cdot,s), \lambda(\cdot,s) \rangle_{H^1_0(\Omega)} \, e^{-s/T} \, \dd s.
\end{equation*}
We define the space--time bilinear form $\A$ as 
\begin{equation} \label{eq:23}
\begin{aligned}
    \A((U,V),(\lambda,\chi)) := & \langle\partial_t V, \lambda \rangle_e + (c^2\nabla_{\bx} U, \nabla_{\bx} \lambda)_{L^2_e(Q_T)}
    \\ & \quad - (\partial_t U, \chi)_{L^2_e(Q_T)} + (V, \chi)_{L^2_e(Q_T)}.
\end{aligned}    
\end{equation}
The \emph{minimal} regularity required for $\mathcal{A}$ to be well-defined is
\begin{equation} \label{eq:24}
\begin{aligned}
    (U,V) & \in \bigl(L^2(0,T; H_0^1(\Omega)) \cap H^1(0,T; L^2(\Omega))\bigr) \times \bigl(L^2(Q_T) \cap H^1(0,T; [H_0^1(\Omega)]')\bigr), 
    \\ (\lambda,\chi) & \in L^2(0,T;H_0^1(\Omega)) \times L^2(Q_T).
\end{aligned}
\end{equation}
\noindent
As in the analysis of the space--time parabolic problem in~\cite{Steinbach2015}, we introduce the \textit{Newton potential} operator $\NOmega : [H_0^1(\Omega)]' \to H_0^1(\Omega)$ defined by 
\begin{equation} \label{eq:25}
    (c^2 \nabla_{\bx} \NOmega U, \nabla_{\bx} V)_{L^2(\Omega)} = \langle U,V \rangle_{H^1_0(\Omega)} \quad \text{for all~} V \in H_0^1(\Omega).
\end{equation}
From Lax-Milgram's lemma,~$\NOmega$ is well-defined. Moreover, $\NOmega$ is an invertible operator and induces norm in~$[H^1_0(\Omega)]'$, which is equivalent to the standard one:
\begin{equation} \label{eq:26}
    \| U \|_{\NOmega}^2 := \langle U, \NOmega U \rangle_{H_0^1(\Omega)} = (c^2 \nabla_{\bx} \NOmega U, \nabla_{\bx} \NOmega U)_{L^2(\Omega)} = \| c \nabla_{\bx} \NOmega U \|^2_{L^2(\Omega)}.
\end{equation}
We have the following lemma.
\begin{lemma} \label{lem:41}
For all~$U \in L^2(\Omega)$,
\begin{equation*}
    \| c \nabla_{\bx} \NOmega U\|_{L^2(\Omega)} \le \frac{C_\Omega}{c_0} \|U\|_{L^2(\Omega)},
\end{equation*}
where~$C_\Omega>0$ is the constant in the Poincar\'e inequality
\begin{equation*}
    \|V\|_{L^2(\Omega)}\le C_\Omega\|\nabla_{\bx} V\|_{L^2(\Omega)} \qquad \text{for all }\, V\in H^1_0(\Omega),
\end{equation*}
and~$c_0$ the lower bound of~$c$ in~\eqref{eq:2}.
\end{lemma}
\begin{proof}
Let us endow~$H^1_0(\Omega)$ with the norm~$\|\nabla_{\bx}\cdot\|_{L^2(\Omega)}$. By the definition of~$\NOmega$ and~$\|\cdot\|_{\NOmega}$, we have
\begin{align*}
    \| c \nabla_{\bx} \NOmega U\|_{L^2(\Omega)}^2
    =\langle U,\NOmega U\rangle_{H^1_0(\Omega)}
     \le \|U\|_{[H^1_0(\Omega)]'}\|\nabla_{\bx} \NOmega U\|_{L^2(\Omega)},
\end{align*}
and therefore
\begin{equation*}
    \| c \nabla_{\bx} \NOmega U\|_{L^2(\Omega)} \le \frac{1}{c_0} \|U\|_{[H^1_0(\Omega)]'}.
\end{equation*}
The result follows from
\begin{equation*}
    \|U\|_{[H^1_0(\Omega)]'}:=\sup_{0\ne V\in H^1_0(\Omega)}\frac{\langle U,V\rangle_{H^1_0(\Omega)}}{\|\nabla_{\bx} V\|_{L^2(\Omega)}}=\sup_{0\ne V\in H^1_0(\Omega)}\frac{(U,V)_{L^2(\Omega)}}{\|\nabla_{\bx} V\|_{L^2(\Omega)}}\le C_\Omega \|U\|_{L^2(\Omega)},
\end{equation*}
where, in the last step, we applied the Cauchy-Schwarz and Poincar\'e inequalities.
\end{proof}
\noindent
Consider now functions depending on space and time. For $U \in L^2(0,T;[H_0^1(\Omega)]')$, we use the notation~$\NOmega U \in L^2(0,T;H_0^1(\Omega))$ to also denote the function defined via
\begin{equation*}
    (c^2 \nabla_{\bx} \NOmega U, \nabla_{\bx} V)_{L^2_e(Q_T)} = \int_0^T \langle U(\cdot,s), V(\cdot,s) \rangle_{H_0^1(\Omega)} \, e^{-s/T} \dd s
\end{equation*}
for all $V \in L^2(0,T;H_0^1(\Omega))$. Moreover, we define for $U \in L^2(0,T;[H_0^1(\Omega)]')$ the norm
\begin{equation}\label{eq:27}
\begin{aligned}
    \| U \|_{\NOmegae}^2 := \int_0^T \|U(\cdot,s)\|^2_{\NOmega} \, e^{-s/T} \, \dd s & = \int_0^T \|c\nabla_{\bx} \NOmega U(\cdot,s)\|^2_{L^2(\Omega)} \, e^{-s/T} \, \dd s
    \\ & =\|c\nabla_{\bx} \NOmega U\|_{L^2_e(Q_T)}^2.
\end{aligned}    
\end{equation}
\noindent
Using Lemma~\ref{lem:41}, we prove the following auxiliary result.
\begin{lemma} \label{lem:42}
For all $U \in L^2(Q_T)$ and $ V \in L^2(0,T; [H_0^1(\Omega)]')$, we have
\begin{equation*}
    \| U \|_{\NOmegae} \leq \frac{C_\Omega}{c_0} \| U \|_{L^2_e(Q_T)}\,, \quad \| \NOmega V \|_{L^2_e(Q_T)} \leq \frac{C_\Omega}{c_0} \| V \|_{\NOmegae},
\end{equation*}
where~$C_\Omega>0$ is the constant in the Poincar\'e inequality, and~$c_0$ the constant in~\eqref{eq:2}.
\end{lemma}
\begin{proof}
For all~$t \in (0,T)$, if $U(\cdot,t) \in L^2(\Omega)$, then $\NOmega U(\cdot,t) \in H^1_0(\Omega)$. Then, using~Lemma~\ref{lem:41} we derive the first bound:
\begin{equation*}
\begin{aligned}
    \| U\|^2_{\NOmegae} &= \int_0^T \langle U(\cdot,t), \NOmega U(\cdot,t) \rangle_{H_0^1(\Omega)} \, e^{-t/T} \, \dd t
     = \int_0^T \| c^2 \nabla_{\bx} \NOmega U(\cdot,t) \|^2_{L^2(\Omega)} \, e^{-t/T} \, \dd t
    \\ & \le \frac{C_\Omega^2}{c_0^2} \int_0^T \| U(\cdot,t) \|^2_{L^2(\Omega)} \, e^{-t/T} \, \dd t
    = \frac{C_\Omega^2}{c_0^2} \| U \|^2_{L^2_e(Q_T)}.
\end{aligned}
\end{equation*}
For the second bound, if $V(\cdot,t) \in [H^1_0(\Omega)]'$, the Poincar\'e inequality implies
\begin{equation*}
    \| \NOmega V(\cdot,t) \|_{L^2(\Omega)} \le C_\Omega \| \nabla_{\bx} \NOmega V(\cdot,t) \|_{L^2(\Omega)},
\end{equation*}
from which we conclude 
\begin{align*}
    \| \NOmega V \|_{L^2_e(Q_T)}^2 &=\int_0^T \| \NOmega V(\cdot,t) \|_{L^2(\Omega)}^2 \, e^{-t/T} \, \dd t
    \\ & \le \frac{C_\Omega^2}{c_0^2} \int_0^T \| c \nabla_{\bx} \NOmega V(\cdot,t) \|_{L^2(\Omega)}^2 \, e^{-t/T} \, \dd t 
    \\&  = \frac{C_\Omega^2}{c_0^2} \int_0^T \langle V(\cdot,t), \NOmega V(\cdot,t)\rangle_{H^1_0(\Omega)} \, e^{-t/T} \, \dd t 
     = \frac{C_\Omega^2}{c_0^2} \| V \|_{\NOmegae}^2.
\end{align*}
\end{proof}
\noindent
For $(U,V)$ and $(\lambda,\chi)$ with regularity as in \eqref{eq:24}, we define the following norms:
\begin{equation} \label{eq:28}
\begin{aligned} 
    \| (U,V) \|^2_{\mathcal{V}_e(Q_T)} & := \| \partial_t U \|_{L^2_e(Q_T)}^2 + \| \partial_t V \|^2_{\NOmega_e} + \| c \nabla_{\bx} U \|^2_{L^2_e(Q_T)} + \| V \|^2_{L^2_e(Q_T)}, 
    \\  \| (\lambda,\chi) \|^2_{\mathcal{W}_e(Q_T)} & := \| \lambda \|^2_{L^2_e(Q_T)} + \| \chi \|^2_{\NOmegae}.
\end{aligned}
\end{equation}
For functions with additional regularity, the bilinear form~$\A$ defined in \eqref{eq:23} satisfies an inf-sup condition with respect to the norms in~\eqref{eq:28}.
\begin{proposition}[Inf-sup in the norms in~\eqref{eq:28}] \label{prop:43}
For any 
\begin{equation} \label{eq:29}
    (U,V) \in H^1_{0,\bullet}(0,T;H_0^1(\Omega)) \times H^1_{0,\bullet}(0,T;L^2(\Omega)), 
\end{equation}
there exists
\begin{equation} \label{eq:30}
    (\lambda,\chi) \in L^2(0,T;H_0^1(\Omega)) \times L^2(Q_T)
\end{equation}
such that
\begin{equation} \label{eq:31}
    \frac{\A((U,V),(\lambda,\chi))}{\| (U,V) \|_{\mathcal{V}_e(Q_T)}\| (\lambda,\chi) \|_{\mathcal{W}_e(Q_T)}} \ge \frac{1}{2\sqrt{C_\Omega^2/c_0^2+4T^2}},
\end{equation}
where~$C_\Omega$ is the Poincar\'e inequality, and~$c_0$ the constant in~\eqref{eq:2}.
\end{proposition}
\begin{proof}
 We prove that, for any~$(U, V)$ with regularity as in \eqref{eq:29}, there exists~$(\chi, \lambda)$ with regularity as in \eqref{eq:30} such that
\begin{align} 
\label{eq:32}
    \|( \lambda, \chi)\|_{\mathcal{W}_e(Q_T)} &\le C_1 \| (U,V) \|_{\mathcal{V}_e(Q_T)},
    \\ \label{eq:33} \A((U,V),(\lambda,\chi))&\ge C_2\| (U,V) \|^2_{\mathcal{V}_e(Q_T)},
\end{align}
with constants~$C_1,C_2>0$ possibly depending on~$T$. Then, the estimate in~\eqref{eq:31} follows with constant on the right-hand smaller than or equal to~$C_2/C_1$.

\noindent
Let us fix~$(U, V)$ with regularity as in \eqref{eq:29} and define~$(\chi, \lambda)$ as
\begin{equation}\label{eq:34}
    \chi := -\partial_t U +2T\, \partial_t V, \quad \quad  
    \lambda := \NOmega \partial_t V + 2T\, \partial_t U.
\end{equation}
Note that $(\lambda,\chi)$ satisfies the regularity in \eqref{eq:30}. Using Lemma \ref{lem:42}, we have 
\begin{align*}
    \| (\lambda,\chi) \|^2_{\mathcal{W}_e(Q_T)} & = \| \lambda \|_{L^2_e(Q_T)}^2 +  \| \chi \|^2_{\NOmegae} \\
     & \leq  \| \NOmega \partial_t V \|_{L^2_e(Q_T)}^2 + 4T^2  \| \partial_t U \|_{L^2_e(Q_T)}^2 + \| \partial_t U\|^2_{\NOmegae} + 4T^2 \| \partial_t V\|^2_{\NOmegae}
    \\ & \le (C_\Omega^2/c_0^2+4T^2) \| (U,V) \|^2_{\mathcal{V}_e(Q_T)}\,,
\end{align*}
which proves~\eqref{eq:32} with~$C_1^2=C_\Omega^2/c_0^2+4T^2$. Using~\eqref{eq:9}, we compute
\begin{equation} \label{eq:35}
\begin{aligned}
    \A((U,V), (\lambda, \chi)) & = \| \partial_t V \|_{\NOmegae}^2 + 2T (\partial_t V, \partial_t U)_{L^2_e(Q_T)} + (U, \partial_t V)_{L^2_e(Q_T)}
    \\ & \quad + 2T  (c^2 \nabla_{\bx} U,\nabla_{\bx} \partial_t U)_{L^2_e(Q_T)} + \| \partial_t U \|_{L^2_e(Q_T)}^2 
    \\ & \quad - 2T (\partial_t U, \partial_t V)_{L^2_e(Q_T)} - (V, \partial_t U)_{L^2_e(Q_T)} 
    + 2T (V, \partial_t V)_{L^2_e(Q_T)}
    \\ & =  \| \partial_t U \|_{L^2_e(Q_T)}^2 + \| \partial_t V \|_{\NOmegae}^2  +\|c\nabla_{\bx} U \|_{L^2_e(Q_T)}^2
    \\ & \quad + \frac{T}{e} \| c\nabla_{\bx} U(\cdot,T)\|_{L^2(\Omega)}^2 +\| V \|_{L^2_e(Q_T)}^2 
    \\ & \quad + \frac{T}{e}\| V(\cdot,T)\|_{L^2(\Omega)}^2 + (U, \partial_t V)_{L^2_e(Q_T)} - (V, \partial_t U)_{L^2_e(Q_T)}.
\end{aligned}
\end{equation}
Recalling the definition of~$\NOmega$ in~\eqref{eq:25} and using the Young inequality, we deduce
\begin{align*}
    (U,\partial_t V)_{L^2_e(Q_T)} & = \int_0^T ( U(\cdot,t), \partial_t V(\cdot,t) )_{L^2(\Omega)} \, e^{-t/T} \, \dd t \\ &= \int_0^T (c^2\nabla_{\bx} U(\cdot,t), \nabla_{\bx} \NOmega \partial_t V(\cdot,t) )_{L^2(\Omega)} \, e^{-t/T} \, \dd t
    \\ & \ge - \frac{1}{2} \| c\nabla_{\bx} U \|^2_{L^2_e(Q_T)} - \frac{1}{2} \| c \nabla_{\bx} \NOmega \partial_t V \|^2_{L^2_e(Q_T)}\\
    & = - \frac{1}{2} \| c\nabla_{\bx} U \|^2_{L^2_e(Q_T)} - \frac{1}{2} \| \partial_t V \|^2_{\NOmegae}. 
\end{align*}
From this and~$(V,\partial_t U)_{L^2_e(Q_T)}\ge -\frac12 \| V \|^2_{L^2_e(Q_T)} -\frac12 \| \partial_t U \|^2_{L^2_e(Q_T)}$,
we continue~\eqref{eq:35} as
\begin{align*}
    \A((U,V), (\lambda, \chi))
    &  \ge \| \partial_t U \|_{L^2_e(Q_T)}^2 +  \| \partial_t V \|_{\NOmegae}^2  +\| c\nabla_{\bx} U \|_{L^2_e(Q_T)}^2\ +\| V \|_{L^2_e(Q_T)}^2
    \\ &  \quad -\frac{1}{2} \|c \nabla_{\bx} U \|_{L^2_e(Q_T)}^2 - \frac{1}{2} \| \partial_t V \|_{\NOmegae}^2  -\frac{1}{2}\| V \|_{L^2_e(Q_T)}^2 - \frac{1}{2}\| \partial_t U \|_{L^2_e(Q_T)}^2 
    \\ & = \frac{1}{2} \|(U,V)\|^2_{\mathcal{V}_e(Q_T)}.
\end{align*}
This leads to~\eqref{eq:33} with~$C_2 = 1/2$. Thus, taking into account that we have proven~\eqref{eq:32} with $C_1=\sqrt{C_\Omega^2/c_0^2+4T^2}$, we obtain \eqref{eq:31}.
\end{proof}
\begin{remark}
The additional regularity assumed in~\eqref{eq:29} gives the
possibility of including
the second terms in the definitions of~$\chi$ 
and~$\lambda$ in~\eqref{eq:34}. Indeed, under only the minimal regularity assumption on~$(U,V)$ given in~\eqref{eq:24}, the terms~$\partial_t V$ and~$\partial_t U$ are not guaranteed to be in~$L^2(Q_T)$ and~$L^2(0,T;H_0^1(\Omega))$, 
which is required by~\eqref{eq:30}. 
\eremk \end{remark}

\noindent
We also derive the following continuity and coercivity-like estimates. 
\begin{proposition} \label{prop:45}
For~$(U,V)$ and~$(\lambda,\chi)$ with regularity as in~\eqref{eq:24}, we have
\begin{equation*}
    \A((U,V),(\lambda,\chi)) \le \sqrt{2} \| (U,V) \|_{\mathcal{V}_e(Q_T)}(\| c\nabla_{\bx} \lambda \|_{L^2_e(Q_T)}^2+\| \chi \|_{L^2_e(Q_T)}^2)^{\frac12},
\end{equation*}
where~$\| \cdot \|_{\mathcal{V}_e(Q_T)}$ is the defined in~\eqref{eq:28}. Furthermore, for $(U,V)$ with regularity as in~\eqref{eq:24} we also have
\begin{equation*}
    \A((U,V),(\partial_t U, \partial_t V)) \ge \frac{1}{2T} ( \|c\nabla_{\bx} U \|^2_{L^2_e(Q_T)} + \| V \|^2_{L^2_e(Q_T)}).
\end{equation*}
\end{proposition}
\begin{proof}
The continuity property follows from the definition of $\NOmega$ in \eqref{eq:25} and the Cauchy-Schwarz inequality: 
\begin{align*}
    \abs{\A((U,V),(\lambda,\chi))} & \le
    \| \partial_t V \|_{\NOmegae} \| c \nabla_{\bx} \lambda \|_{L^2_e(Q_T)} +\|c \nabla_{\bx} U \|_{L^2_e(Q_T)} \| c \nabla_{\bx} \lambda \|_{L^2_e(Q_T)} \\ &\qquad  + \| \partial_t U \|_{L^2_e(Q_T)}\| \chi \|_{L^2_e(Q_T)}+\| V \|_{L^2_e(Q_T)}\| \chi \|_{L^2_e(Q_T)}
    \\ & \le \sqrt{2} \| (U,V) \|_{\mathcal{V}_e(Q_T)}(\| c\nabla_{\bx} \lambda \|_{L^2_e(Q_T)}^2+\| \chi \|_{L^2_e(Q_T)}^2)^{\frac12}.
\end{align*}
To prove the coercivity-like property, we apply~\eqref{eq:9}:
\begin{align*}
    \A((U,V),(\partial_t U,\partial_t V)) & = (\partial_t V, \partial_t U)_{L^2_e(Q_T)} + (c^2\nabla_{\bx} U, \nabla_{\bx} \partial_t U)_{L^2_e(Q_T)} \\ &\qquad -(\partial_t U, \partial_t V)_{L^2_e(Q_T)} + (V,\partial_t V)_{L^2_e(Q_T)}
    \\ & = (c^2 \nabla_{\bx} U, \nabla_{\bx} \partial_t U)_{L^2_e(Q_T)} + (V,\partial_t V)_{L^2_e(Q_T)}
    \\ & \ge \frac{1}{2T} (\|c \nabla_{\bx} U \|_{L^2_e(Q_T)}^2 + \| V \|_{L^2_e(Q_T)}^2).
\end{align*}
\end{proof}
\begin{remark}\label{rem:46}
We emphasize that, for the bilinear form~$\A$, we established continuity using a stronger norm for the test functions, as compared to the norm employed in the inf-sup condition~\eqref{eq:31}. This leads to difficulties in the analysis of both the variational formulation and its discretization.
\end{remark}

\subsection{Space--time variational formulation}
Finally, we consider the variational formulation
\begin{tcolorbox}[
    colframe=black!50!white,
    colback=blue!5!white,
    boxrule=0.5mm,
    sharp corners,
    boxsep=0.5mm,
    top=0.5mm,
    bottom=0.5mm,
    right=0.25mm,
    left=0.1mm
]
    \begingroup
    \setlength{\abovedisplayskip}{0pt}
    \setlength{\belowdisplayskip}{0pt}
    \phantom{1} find $(U,V) \in H^1_{0,\bullet}(0,T;H_0^1(\Omega)) \times H^1_{0,\bullet}(0,T;L^2(\Omega))$ such that:  \vspace{0.2cm}
    \begin{equation} \label{eq:36}
        \A((U,V),(\lambda,\chi))=(F,\lambda)_{L^2_e(Q_T)} - (c^2 \nabla_{\bx} U_0, \nabla_{\bx} \lambda)_{L^2_e(Q_T)} - (V_0,\chi)_{L^2_e(Q_T)}
        \vspace{0.2cm}
    \end{equation}
    \phantom{1} for all~$(\lambda,\chi) \in L^2(0,T;H_0^1(\Omega)) \times L^2(Q_T)$.
    \endgroup
\end{tcolorbox}
\medskip
\noindent
As a consequence of Remark~\ref{rem:46}, for the analysis of problem~\eqref{eq:36}, Ne\v{c}as' theorem cannot be applied, and we rely on the existence result in Proposition~\ref{prop:22}.

\begin{theorem} \label{thm:47}
Let us assume the regularity on the data as in Assumption~\ref{ass:21}. Then, there exists a unique solution to problem \eqref{eq:36}, and it satisfies
\begin{equation*}
\begin{split}
    &\| (U,V) \|_{\mathcal{V}_e(Q_T)} \\
    &\quad\le 2\sqrt{C_\Omega^2/c_0^2 + 4T^2}\, \bigl(\| F\|_{L^2_e(Q_T)} +\sqrt{T}\,\| \nabla_{\bx} \cdot (c \nabla_{\bx} U_0)\|_{L^2(\Omega)} +\sqrt{T}\, \| c\nabla_{\bx} V_0 \|_{L^2(\Omega)}\bigr).
    \end{split}
\end{equation*}
Moreover, it holds true that~$V \in L^2(0,T;H_0^1(\Omega))$.
\begin{proof}
The existence of a solution $U$ to the variational formulation~\eqref{eq:4}, with the required regularity, follows from Assumption~\ref{ass:21}, see Proposition~\ref{prop:22}. Setting $V = \partial_t U$, this implies the existence of a solution to~\eqref{eq:36}, together with the property~$V \in L^2(0,T;H_0^1(\Omega))$. 
The corresponding stability estimate follows from the inf-sup condition in Proposition~\ref{prop:43}, together with the following continuity property of the functional on the right-hand side:
\begin{equation*} \label{eq:Fcont}
\begin{aligned}
    &(F,\lambda)_{L^2_e(Q_T)} - (c^2 \nabla_{\bx} U_0, \nabla_{\bx} \lambda)_{L^2_e(Q_T)} - (V_0,\chi)_{L^2_e(Q_T)}
    \\ 
    &\quad =(F,\lambda)_{L^2_e(Q_T)} + (\nabla_{\bx}\cdot(c^2 \nabla_{\bx}U_0), \lambda)_{L^2_e(Q_T)} - (c\nabla_{\bx} V_0,c\nabla_{\bx} \NOmega\chi)_{L^2_e(Q_T)}
    \\  
    &\quad \le \left(\| F \|_{L^2_e(Q_T)}  +\sqrt{T}\,\| \nabla_{\bx} \cdot (c \nabla_{\bx} U_0)\|_{L^2(\Omega)}  +\sqrt{T}\, \| c\nabla_{\bx} V_0 \|_{L^2(\Omega)}\right) \|(\lambda,\chi) \|_{\mathcal{W}_e(Q_T)},
\end{aligned}
\end{equation*}
where, in the last step, we applied the Cauchy-Schwarz inequality and used the identity $\|c\nabla_{\bx} \NOmega\chi\|_{L^2_e(Q_T)}=\|\chi\|_{\NOmega_e}$, see~\eqref{eq:27}.
Indeed, we compute
\begin{align*}
    &\frac{1}{2\sqrt{C_\Omega^2/c_0^2+4T^2}} \| (U,V) \|_{\mathcal{V}_e(Q_T)} 
     \le
    \sup_{0 \ne (\lambda,\chi) \in (L^2(0,T;H_0^1(\Omega)) \times L^2(Q_T))}\frac{\A((U,V),(\lambda,\chi))}{\|(\lambda,\chi) \|_{\mathcal{W}_e(Q_T)}}
    \\ &\,  = 
    \sup_{0 \ne (\lambda,\chi) \in L^2(0,T;H_0^1(\Omega)) \times L^2(Q_T)}\frac{(F,\lambda)_{L^2_e(Q_T)} - (c^2 \nabla_{\bx}U_0,\nabla_{\bx} \lambda)_{L^2_e(Q_T)} - (V_0,\chi)_{L^2_e(Q_T)}}{\|(\lambda,\chi) \|_{\mathcal{W}_e(Q_T)}}
     \\ &\, \le \| F\|_{L^2_e(Q_T)} +\sqrt{T}\,\| \nabla_{\bx} \cdot (c \nabla_{\bx} U_0)\|_{L^2(\Omega)} +\sqrt{T}\, \| c\nabla_{\bx} V_0 \|_{L^2(\Omega)}.
\end{align*}
\end{proof}
\end{theorem}

\section{Discretization of the space--time problem} \label{sec:5}

In this section, we study conforming space--time tensor product discretizations of the variational formulation in~\eqref{eq:36}, proving their unconditional stability and deriving error estimates.

\subsection{Definition of the method and well-posedness}

For the spatial discretization, let us consider a discrete space~$S_{h_{\bx}}(\Omega) \subset H_0^1(\Omega)$ depending on a spatial parameter~$h_{\bx}$ (e.g. piecewise linear, continuous functions over a triangulation of~$\Omega$ of mesh size~$h_{\bx}$). For the temporal discretization, we introduce a discrete space~$S_{h_t}(0,T) \subset H_{0,\bullet}^1(0,T)$ on a mesh of~$(0,T)$ with mesh size~$h_t$. Finally, we define the tensor product spaces
\begin{align*}
    Q_{\bh}(Q_T):=S_{h_{\bx}}(\Omega) \otimes S_{h_t}(0,T), \quad\quad  \partial_t Q_{\bh} (Q_T) :=S_{h_{\bx}}(\Omega) \otimes~\partial_t S_{h_t}(0,T).
\end{align*}
The conforming discretization of \eqref{eq:36} we consider reads as follows:
\begin{tcolorbox}[
    colframe=black!50!white,
    colback=blue!5!white,
    boxrule=0.5mm,
    sharp corners,
    boxsep=0.5mm,
    top=0.5mm,
    bottom=0.5mm,
    right=0.25mm,
    left=0.1mm
]
    \begingroup
    \setlength{\abovedisplayskip}{0pt}
    \setlength{\belowdisplayskip}{0pt}
    \phantom{1} find $ (U_{\bh},V_{\bh}) \in (Q_{\bh}(Q_T))^2$ such that:  \vspace{0.2cm}
    \begin{equation} \label{eq:37}
    \begin{aligned}
        \A((U_{\bh},V_{\bh}),(\lambda_{\bh},\chi_{\bh})) =(F,\lambda_{\bh})_{{L^2_e(Q_T)}} & -(c^2 \nabla_{\bx} U_{0}, \nabla_{\bx} \lambda_{\bh})_{{L^2_e(Q_T)}} 
        \\ & - (V_0,\chi_{\bh})_{{L^2_e(Q_T)}} 
    \end{aligned}
    \end{equation}
    \phantom{1} for all~$(\lambda_{\bh}, \chi_{\bh}) \in (\partial_t Q_{\bh} (Q_T))^2$.
    \endgroup
\end{tcolorbox}
\medskip

\noindent
Let~$\NOmegahx : L^2(\Omega) \to S_{h_{\bx}}(\Omega)$ be the \emph{discrete} Newton potential operator defined by
\begin{equation} \label{eq:38}
    (c^2 \nabla_{\bx} \NOmegahx U, \nabla_{\bx} V_{h_{\bx}})_{L^2(\Omega)} = (U,V_{h_{\bx}})_{L^2(\Omega)} \quad \text{for all~} V_{h_{\bx}} \in S_{h_{\bx}}(\Omega).
\end{equation}
Owing to Lax-Milgram's lemma, the operator $\NOmegahx$ is well-defined. Proceeding as in Lemma \ref{lem:41}, we deduce for all $U \in L^2(\Omega)$
\begin{equation*}
    \| c \nabla_{\bx} \NOmegahx U \|_{L^2(\Omega)} \le \frac{C_\Omega}{c_0} \|U \|_{L^2(\Omega)}.
\end{equation*}
For functions in $L^2(\Omega)$, similarly as in \eqref{eq:26}, we define $\| U \|_{\NOmegahx} := \| c \nabla_{\bx} \NOmegahx U \|_{L^2(\Omega)}$. Note that this is a seminorm in $L^2(\Omega)$, and is actually a norm in $S_{h_{\bx}}(\Omega)$. Indeed, for $U_{h_{\bx}} \in S_{h_{\bx}}(\Omega)$, if $\| U_{h_{\bx}} \|_{\NOmegahx} = 0$, then $\NOmegahx U_{h_{\bx}} = 0$ and from \eqref{eq:38} 
\begin{equation*}
    (U_{h_{\bx}}, V_{h_{\bx}})_{L^2(\Omega)} = 0 \quad \text{for all~} V_{h_{\bx}} \in S_{h_{\bx}}(\Omega).
\end{equation*}
that implies $U_{h_{\bx}} =0$. Moreover,  we notice that, for all $U \in L^2(\Omega)$,
\begin{equation}\label{eq:39}
    \| U \|_{\NOmegahx} \le \| U \|_{\NOmega}. 
\end{equation}
Indeed, with \eqref{eq:38}, \eqref{eq:25} and Cauchy-Schwarz inequality, we compute
\begin{align*}
    \| U \|^2_{\NOmegahx} &= (c^2\nabla_{\bx} \NOmegahx U, \nabla_{\bx} \NOmegahx U)_{L^2(\Omega)} = (U, \NOmegahx U)_{L^2(\Omega)}  
    \\ & = (c^2\nabla_{\bx} \NOmega U, \nabla_{\bx} \NOmegahx U)_{L^2(\Omega)} 
    \le \| c\nabla_{\bx} \NOmega U \|_{L^2(\Omega)} \| c\nabla_{\bx} \NOmegahx U \|_{L^2(\Omega)}
    \\ & = \| U \|_{\NOmega} \| U \|_{\NOmegahx}.
\end{align*}

\noindent
For space--time functions~$U \in L^2(Q_T)$, we define $\NOmega_{h_{\bx}} U \in S_{h_{\bx}}(\Omega) \otimes L^2(0,T)$ by
\begin{equation*}
    (c^2 \nabla_{\bx} \NOmegahx U, \nabla_{\bx} V_{h_{\bx}})_{L^2_e(Q_T)} = \int_0^T ( U(\cdot,s), V_{h_{\bx}}(\cdot,s) )_{L^2(\Omega)} \, e^{-s/T} \dd s
\end{equation*}
for all $V_{h_{\bx}} \in S_{h_{\bx}}(\Omega) \otimes L^2(0,T)$, and the seminorm in $L^2(Q_T)$ (norm in $S_{h_{\bx}}(\Omega) \otimes L^2(0,T)$)
\begin{equation*}
    \| U \|_{\NOmegaeh}^2 := \int_0^T \| c \nabla_{\bx}  \NOmegahx U(\cdot,s) \|_{L^2(\Omega)}^2 \, e^{-s/T} \, \dd s.
\end{equation*}
We summarize in the following lemma auxiliary results analogous to those in Lemma~\ref{lem:42}.
\begin{lemma} \label{lem:51}
For all~$U, V\in L^2(Q_T)$, we have
\begin{equation*}
    \| U \|_{\NOmegaeh} \leq \frac{C_\Omega}{c_0} \| U \|_{{L^2_e(Q_T)}}\, , \quad \| \NOmegahx V \|_{{L^2_e(Q_T)}} \leq \frac{C_\Omega}{c_0} \| V \|_{\NOmegaeh},
\end{equation*}
where~$C_\Omega>0$ is the constant in the Poincar\'e inequality, and~$c_0$ the constant in~\eqref{eq:2}.
\end{lemma}
\noindent
For $(U_{\bh}, V_{\bh}) \in (Q_{\bh}(Q_T))^2$ and $(\lambda_{\bh}, \chi_{\bh}) \in (\partial_t Q_{\bh}(Q_T))^2$, we define the following discrete norms
{\small\begin{equation} \label{eq:40}
\begin{aligned} 
    \| (U_{\bh},V_{\bh}) \|^2_{\mathcal{V}_{e,\bh}(Q_T)} & := \| \partial_t U_{\bh} \|_{{L^2_e(Q_T)}}^2 + \| \partial_t V_{\bh} \|^2_{\NOmega_{e,\bh}}  + \| c \nabla_{\bx} U_{\bh} \|^2_{{L^2_e(Q_T)}} + \| V_{\bh} \|^2_{{L^2_e(Q_T)}}, 
    \\  \| (\lambda_{\bh},\chi_{\bh}) \|^2_{\mathcal{W}_{e,\bh}(Q_T)} & := \| \lambda_{\bh} \|^2_{{L^2_e(Q_T)}} + \| \chi_{\bh} \|^2_{\NOmegaeh}.
\end{aligned}
\end{equation}}
\begin{proposition}[Inf-sup in the discrete norms \eqref{eq:40}] \label{prop:52}
For all $(U_{\bh},V_{\bh}) \in (Q_{\bh}(Q_T))^2$ there exists $(\lambda_{\bh},\chi_{\bh}) \in (\partial_t Q_{\bh}(Q_T))^2$ such that
\begin{equation*}
    \frac{\A((U_{\bh},V_{\bh}),(\lambda_{\bh},\chi_{\bh}))}{\| (U_{\bh},V_{\bh}) \|_{\mathcal{V}_{e,\bh}(Q_T)}\| (\lambda_{\bh},\chi_{\bh}) \|_{\mathcal{W}_{e,\bh}(Q_T)}} \ge \frac{1}{2\sqrt{C_\Omega^2/c_0^2+4T^2}}
\end{equation*}
where~$C_\Omega$ is the constant in the Poincar\'e inequality.
\end{proposition}
\begin{proof}
The proof is identical to that of Proposition~\ref{prop:43}, using the bounds in Lemma~\ref{lem:51}. 
The test functions to be used here are
\begin{equation*}
    \chi_{\bh} = - \partial_t U_{\bh} + 2T \partial_t V_{\bh}, \quad \lambda_{\bh} = \NOmegahx \partial_t V_{\bh} + 2T \partial_t U_{\bh}.
\end{equation*}
The result then follows by the same argument as in Proposition~\ref{prop:43}.
\end{proof}
\noindent
As a consequence, we have well-posedness of the discrete formulation in~\eqref{eq:37}. 
\begin{corollary} \label{cor:53}
Let the source satisfy $F \in L^2(Q_T)$. Then, there exists a unique solution to problem \eqref{eq:37}, and it satisfies
\begin{equation*}
    \| (U_{\bh},V_{\bh}) \|_{\mathcal{V}_{e,\bh}} \le 2\sqrt{C_\Omega^2/c_0^2 + 4T^2} \bigl(\| F\|_{{L^2_e(Q_T)}} +\| \nabla_{\bx} \cdot (c \nabla_{\bx} U_0)\|_{L^2(\Omega)} + \| c\nabla_{\bx} V_0\|_{L^2(\Omega)}\bigr).
\end{equation*}
\end{corollary}
\begin{proof}
The existence of solutions and the stability estimate follow from the inf-sup condition of Proposition~\ref{prop:52}, while uniqueness follows from existence and the fact that~\eqref{eq:37} is a square, finite dimensional linear system. The stability estimate is derived exactly as in Theorem \ref{thm:47}.
\end{proof}
\begin{remark}
The use of exponential weights allows us to prove stability for a general class of conforming tensor product discrete spaces, under the minimal assumption that the trial space is the time derivative of the test space. From a practical point of view, the presence of the exponential weights in the discrete formulation may be unnecessary. This is the case for continuous piecewise polynomial approximations in time~\cite{FrenchPeterson1996,Walkington2014,DongMascottoWang2024, Gomez2025}, where discrete stability is proven by selecting test functions containing a weight that mimics an exponential function. Whether this strategy can be extended to more general discrete spaces remains an open question.
\eremk \end{remark}

\subsection{Error estimates}

\noindent
In this section, we derive error estimates for the solution to~\eqref{eq:37} in the norm~$\| \cdot \|_{\mathcal{V}_{e,\bh}(Q_T)}$ defined in~\eqref{eq:40}. We observe that the norm~$\| \cdot \|_{\mathcal{V}_{e}(Q_T)}$ and the discrete norm~$\| \cdot \|_{\mathcal{V}_{e,\bh}(Q_T)}$ differ solely in the terms~$\|\cdot\|_{\NOmegae}$ and~$\|\cdot\|_{\NOmega_{e,\bh}}$, while all other terms coincide. In this analysis, we use similar techniques as in~\cite{DongMascottoWang2024} and~\cite{Gomez2025}.

\noindent
We start by introducing projector operators and their properties in Section~\ref{sec:521}, then we derive convergence rates in Section~\ref{sec:522} under standard assumptions on the discrete spaces.

\subsubsection{Projection operators} \label{sec:521}
We introduce the following
projection operators:\medskip

\begin{itemize}[leftmargin=1.5em]
    \item (elliptic projector in space)
$\Pg : L^2(0,T;H^1(\Omega)) \to S_{h_{\bx}}(\Omega) \otimes L^2(0,T)$ defined by
\begin{equation} \label{eq:41}
    (c^2 \nabla_{\bx} (\Pg - \Id) W, \nabla_{\bx} W_{h_{\bx}})_{L_e^2(Q_T)} = 0 \quad \text{for all~} W_{h_{\bx}} \in S_{h_{\bx}}(\Omega) \otimes L^2(0,T),
\end{equation}
\item (elliptic projector in time)
$\Pt : H^1(0,T;L^2(\Omega)) \to L^2(\Omega) \otimes S_{h_t}(0,T)$ defined by 
\begin{equation} \label{eq:42}
    (\partial_t (\Pt-\Id) W, \chi_{h_t})_{L^2_e(Q_T)} = 0 \quad \text{for all~} \chi_{h_t} \in L^2(\Omega) \otimes \partial_t S_{h_t}(0,T).
\end{equation}
\end{itemize}
\begin{remark}
We note that the right-hand side of~\eqref{eq:37} can also be written as
\begin{equation*}
    (\Pi^0_{\bh} F,\lambda_{\bh})_{L^2_e(Q_T)} - (c^2 \nabla_{\bx} (\Pg U_{0}), \nabla_{\bx} \lambda_{\bh})_{L^2(\Omega)} - (\Px V_0,\chi_{\bh})_{L^2(\Omega)},
\end{equation*}
where~$\Pi^0_{\bh}$ denotes the $L^2_e(Q_T)$ projection operator into~$\partial_t Q_{\bh}(Q_T)$, and~$\Px$ denotes the~$L^2(\Omega)$ projection operator into~$S_{h_{\bx}}(\Omega)$.
\eremk \end{remark}
\noindent
In the following lemma, we collect useful properties of these operators (see, e.g., \cite[\S 2]{AzizMonk1989} and \cite[Lemma 10]{SteinbachZank2019}).
\begin{lemma} \label{lem:55}
The following composition of projectors is well defined and commute:
\begin{align*}
    \Pg \Pt & = \Pt \Pg : H^1(0,T;H^1(\Omega)) \to Q_{\bh}(Q_T).
\end{align*}
Moreover, the differential operators and projectors commute:
\begin{align*}
    \partial_t \Pg W  = \Pg \partial_t W, \quad\quad \nabla_{\bx} \Pt W = \Pt \nabla_{\bx} W \quad\quad\quad \text{for all~} W \in H^1(0,T;H^1(\Omega)).
\end{align*}
Furthermore, the projectors satisfy standard approximation properties:
\begin{itemize}[leftmargin=1.5em]
\item 
for all $W \in L^2(0,T;H^1(\Omega))$,
\begin{equation} \label{eq:43}
    \| c \nabla_{\bx} (\Pg -\Id) W\|_{L_e^2(Q_T)} = \inf_{W_{h_{\bx}} \in S_{h_{\bx}}(\Omega) \otimes L^2(0,T)} \| c\nabla_{\bx} (W_{h_{\bx}} - W) \|_{L_e^2(Q_T)} \,\,\,  
\end{equation}
\item 
for all $W \in H^1(0,T;L^2(\Omega))$,
\begin{equation} \label{eq:44}
    \| \partial_t (\Pt- \Id) W \|_{L^2_e(Q_T)} = \inf_{W_{h_t} \in L^2(\Omega) \otimes S_{h_t}(0,T)} \| \partial_t(W_{h_t} - W) \|_{L_e^2(Q_T)}. 
\end{equation}
\end{itemize}
Finally, we have the following stability estimates: 
\begin{align}
    \| c \nabla_{\bx} \Pg W\|_{L_e^2(Q_T)} & \le 
    \| c\nabla_{\bx} W \|_{L_e^2(Q_T)} \quad\,\,\,\, \text{for all~} W \in L^2(0,T;H^1(\Omega)), \label{eq:45}
    \\ \| \partial_t \Pt W \|_{{L^2_e(Q_T)}} & \le 
    \| \partial_t W \|_{{L^2_e(Q_T)}} \quad\ \ \,\,\,\,\,  \text{for all~} W \in H^1(0,T;L^2(\Omega)). \label{eq:46}
\end{align}
\end{lemma}
\noindent
As a consequence of the inf-sup condition established in Proposition~\ref{prop:43}, we derive the following auxiliary result.
\begin{lemma} \label{lem:56}
Assume the regularity on the data as in Assumption~\ref{ass:21}. Let~$(U,V)$ be the unique solution to \eqref{eq:36}, whose regularity is stated in Proposition~\ref{prop:22}, and let~$(U_{\bh},V_{\bh})$ be the unique discrete solution to \eqref{eq:37}. Suppose also that
\begin{equation} \label{eq:47}
    \nabla_{\bx} \cdot (c^2 \nabla_{\bx} U) \in H^1(0,T;L^2(\Omega)).
\end{equation}
Then, it holds 
\begin{align*}
    \| (\Pg \Pt U- U_{\bh}, \Pg \Pt V - V_{\bh})\|_{\mathcal{V}_{e,\bh}(Q_T)} \le 
     & \beta\bigl(\| (\Pg-\Id) \partial_t V \|_{{L^2_e(Q_T)}} 
    \\ & \hspace{-5cm} +  \| (\Pt - \Id) \nabla_{\bx} \cdot (c^2 \nabla_{\bx} U) \|_{L^2_e(Q_T)} 
    + \| c \nabla_{\bx} \Pg (\Pt - \Id) V \|_{L^2_e(Q_T)}\bigr),
\end{align*}
where $\beta := 2\sqrt{C_\Omega^2/c_0^2 + 4T^2}$ is the reciprocal of the inf-sup constant in Proposition \ref{prop:52}.
\end{lemma}
\begin{proof}
We use Proposition \ref{prop:52} and Galerkin orthogonality to first obtain, for any $(W_{\bh},Z_{\bh}) \in (Q_{\bh}(Q_T))^2$,
\begin{equation} \label{eq:48}
\begin{aligned}
    & \| (W_{\bh}-U_{\bh},Z_{\bh}-V_{\bh})\|_{\mathcal{V}_{e,\bh}(Q_T)} 
    \\ & \quad\quad \le \beta \sup_{0\ne(\lambda_{\bh},\chi_{\bh}) \in (\partial_t Q_{\bh}(Q_T))^2} \frac{\A((W_{\bh}- U_{\bh},Z_{\bh} - V_{\bh}),(\lambda_{\bh},\chi_{\bh}))}{\| (\lambda_{\bh},\chi_{\bh}) \|_{\mathcal{W}_{e,\bh}(Q_T)}} 
    \\ & \quad\quad = \beta \sup_{0\ne(\lambda_{\bh},\chi_{\bh}) \in (\partial_t Q_{\bh}(Q_T))^2} \frac{\A((W_{\bh}- U,Z_{\bh} - V),(\lambda_{\bh},\chi_{\bh}))}{\| (\lambda_{\bh},\chi_{\bh}) \|_{\mathcal{W}_{e,\bh}(Q_T)}} 
\end{aligned}
\end{equation}
where $\beta = 2\sqrt{C_\Omega^2/c_0^2 + 4T^2}$. Using the definition of~$\A$ in \eqref{eq:23}, we write
\begin{equation} \label{eq:49}
\begin{aligned}
    &  \hspace{-2cm} \A((W_{\bh}- U, Z_{\bh} - V),(\lambda_{\bh},\chi_{\bh})) 
    \\ & = (\partial_t (Z_{\bh} - V), \lambda_{\bh})_{L^2_e(Q_T)}  + (c^2\nabla_{\bx} (W_{\bh}-U), \nabla_{\bx} \lambda_{\bh})_{L^2_e(Q_T)} 
    \\ & \quad - (\partial_t (W_{\bh}-U), \chi_{\bh})_{L^2_e(Q_T)} + ((Z_{\bh}-V), \chi_{\bh})_{L^2_e(Q_T)}
    \\ &  =: I_1 + I_2 + I_3 + I_4.
\end{aligned}
\end{equation}
Now, we take~$W_{\bh} = \Pi^\nabla_{h_{\bx}} \Pt U$ and $Z_{\bh} = \Pi^\nabla_{h_{\bx}} \Pt V$. Then, using the commutativity properties in Lemma~\ref{lem:55}, together with~\eqref{eq:42}, we obtain 
\begin{equation*}
\begin{aligned}
    I_1 &= (\partial_t (\Pg \Pt V - V), \lambda_{\bh})_{L^2_e(Q_T)}  = (\partial_t (\Pt \Pg V - V), \lambda_{\bh})_{L^2_e(Q_T)}
    \\ & = (\partial_t ( \Pg V - V), \lambda_{\bh})_{L^2_e(Q_T)} 
     = ((\Pg - \Id) \partial_t V), \lambda_{\bh})_{L^2_e(Q_T)},
\end{aligned}
\end{equation*}
from which, using the Cauchy-Schwarz inequality, we obtain
\begin{equation} \label{eq:50}
\begin{aligned}
    I_1 & 
     \le \|(\Pg - \text{Id}) \partial_t V \|_{L^2_e(Q_T)} \| (\lambda_{\bh},\chi_{\bh})\|_{\mathcal{W}_{e,\bh}(Q_T)}. 
\end{aligned}
\end{equation}
Moreover, with~\eqref{eq:41} and integration by parts, where the regularity assumption in~\eqref{eq:47} is used, we deduce
\begin{align*}
    I_2 &= (c^2 \nabla_{\bx} (\Pi^\nabla_{h_{\bx}} \Pt U - U), \nabla_{\bx} \lambda_{\bh})_{L^2_e(Q_T)}  = (c^2 \nabla_{\bx} (\Pt U - U), \nabla_{\bx} \lambda_{\bh})_{L^2_e(Q_T)}
    \\ & = - (\nabla_{\bx} \cdot(c^2 \nabla_{\bx} (\Pt U - U)),\lambda_{\bh})_{L^2_e(Q_T)}
     = - ((\Pt - \text{Id})\nabla_{\bx} \cdot(c^2 \nabla_{\bx} U), \lambda_{\bh})_{L^2_e(Q_T)},
\end{align*}
where, in the last step, we used the commutativity properties in Lemma~\ref{lem:55}. Using again the Cauchy-Schwarz inequality, we obtain
\begin{equation} \label{eq:51}
\begin{aligned}
    I_2 \le \|(\Pt - \text{Id}) \nabla_{\bx} \cdot (c^2 \nabla_{\bx} U) \|_{L^2_e(Q_T)} \| (\lambda_{\bh},\chi_{\bh})\|_{\mathcal{W}_{e,\bh}(Q_T)}. 
     \end{aligned}
\end{equation}
For~$I_3$, starting as for~$I_1$ and using that $(\partial_t U,\chi_{\bh})_{L^2_e(Q_T)} = (V,\chi_{\bh})_{L^2_e(Q_T)}$, we have
\begin{equation*}
\begin{aligned}
    I_3  
    = - ((\Pg  - \Id) \partial_t U, \chi_{\bh})_{L^2_e(Q_T)}
    = - ((\Pg  - \Id) V, \chi_{\bh})_{L^2_e(Q_T)}.
    \end{aligned}
\end{equation*}
We proceed by estimating~$I_3 + I_4$ jointly, which allows to take advantage of the cancellation of the terms~$(V, \chi_{\bh})_{L^2_e(Q_T)}$. Using the definition \eqref{eq:38} of~$\mathcal{N}_{h_{\bx}}^\Omega$, together with the Cauchy-Schwarz inequality, gives
\begin{align*}
    I_3 + I_4 = (\Pg (\Pt - \Id) V, \chi_{\bh})_{L^2_e(Q_T)} & = (c^2 \nabla_{\bx} \Pg (\Pt - \Id) V, \nabla_{\bx} \mathcal{N}_{h_{\bx}}^\Omega \chi_{\bh})_{L^2_e(Q_T)}
    \\ &  \le \|c \nabla_{\bx} \Pg (\Pt - \Id) V \|_{L^2_e(Q_T)} \| \chi_{\bh} \|_{\NOmegaeh},
\end{align*}
from which
\begin{equation} \label{eq:52}
    I_3 + I_4 \le \|c \nabla_{\bx} \Pg (\Pt - \Id) V \|_{L^2_e(Q_T)} \| (\lambda_{\bh},\chi_{\bh}) \|_{\mathcal{W}_{e,\bh}(Q_T)}.
\end{equation}
By substituting the estimates~\eqref{eq:50}--\eqref{eq:52} for~$I_1$--$I_4$ into~\eqref{eq:49}, and combining with~\eqref{eq:48}, we obtain the desired claim.
\end{proof}

\subsubsection{Error estimates in the \texorpdfstring{$\mathcal{V}_{e,\bh}(Q_T)$}{V} norm} \label{sec:522}

\noindent
In the following, the notation $Q_1 \apprle Q_2$ (resp. $Q_1 \apprge Q_2$) means that $Q_1$ is bounded above (resp. below) by $\kappa \,Q_2$, where $\kappa>0$ is a constant independent of the mesh sizes $h_t$ and $h_{\bx}$, as well as of the involved functions.
We adopt the notation 
\begin{equation*}
    H_{0,\bullet}^{s_t}(0,T) := H_{0,\bullet}^1(0,T) \cap H^{s_t}(0,T), \quad  H_0^{s_{\bx}}(\Omega) := H_0^1(\Omega) \cap H^{s_{\bx}}(\Omega), \quad \text{for~} s_t, s_{\bx} \ge 1.
\end{equation*}
We proceed by deriving explicit error estimates under the following assumptions on the discrete spaces.
\begin{assumption}[space discretization] \label{ass:57}
The finite dimensional space~$S_{h_{\bx}}(\Omega) \subset H_0^1(\Omega)$ satisfies the following approximation property for some $p_{\bx} \in \mathbb{N}$, $p_{\bx} \ge 1$: 
there exists an operator~$Q_{h_{\bx}}^{p_{\bx}}: H_0^1(\Omega)\to S_{h_{\bx}}(\Omega)$ such that, for all $W\in H_0^{s_{\bx}+1}(\Omega)$ with~$1 \le s_{\bx} \le p_{\bx}$, it satisfies
\begin{align*}
     \| \nabla_{\bx} (Q_{h_{\bx}}^{p_{\bx}} W  - W) \|_{L^2(\Omega)} & \apprle  h_{\bx}^{s_{\bx}} \| W \|_{H^{s_{\bx}+1}(\Omega)}.
\end{align*}
\end{assumption}
\begin{assumption}[time discretization] \label{ass:58}
The finite dimensional space~$S_{h_t}(0,T) \subset H_{0,\bullet}^1(0,T)$ satisfies the following approximation property for some $p_t \in \N$, $p_t \ge 1$: there exists an operator $Q_{h_t}^{p_t} : H_{0,\bullet}^1(0,T) \to S_{h_t}(0,T)$ such that, for all $w \in H_{0,\bullet}^{s_t+1}(0,T)$ with $1 \le s_t \le p_t$, it satisfies
\begin{align*}
     \| \partial_t (Q_{h_t}^{p_t} w - w) \|_{L^2(0,T)} & \apprle  h_t^{s_t} \| \partial_t^{s_t+1} w \|_{L^2(0,T)}.
\end{align*}
\end{assumption}
\begin{remark}
In the case of discrete spaces consisting of continuous piecewise polynomials on shape-regular, simplicial meshes of~$\Omega$ and arbitrary meshes of~$(0,T)$, Assumptions~\ref{ass:57} and~\ref{ass:58} are both satisfied, e.g., with the quasi-interpolation operator defined in Section \cite[\S 22.3]{ErnGuermond2021a}; see \cite[Theorem 22.6]{ErnGuermond2021a}.
For one-dimensional B-splines, Assumption~\ref{ass:58} is satisfied, e.g., with the quasi-interpolant in \cite[Theorem 49]{LycheManniSpeleers2018} (see also references therein). As in the subsequent analysis we do not use inverse estimates, the space and time meshes are not required to be quasi-uniform.
\eremk \end{remark}
\noindent
Under Assumption~\ref{ass:58}, we have the following approximation result for the operator~$\Pt$. 
\begin{lemma} \label{lem:510}
For all $W \in H_{0,\bullet}^{s_t+1}(0,T;L^2(\Omega))$, with~$1 \le s_t \le p_t$ and~$p_t$ as in Assumption~\ref{ass:58}, it holds true that
\begin{equation*}
    \| ( \Pt -\Id) W \|_{L^2_e(Q_T)} \apprle h^{s_t+1}_t \| \partial_t^{s_t+1} W \|_{L_e^2(Q_T)}.
\end{equation*}
\end{lemma}
\begin{proof}
We note that 
\begin{equation}\label{eq:53}
    \| ( \Pt-\Id)W \|_{L^2_e(Q_T)}^2 = \int_\Omega \|(\Pt-\Id) W(\bx,\cdot)\|^2_{L^2_e(0,T)} \dd \bx.
\end{equation}
We can apply the \emph{Aubin--Nitsche trick} to derive the following estimate:
\begin{equation}\label{eq:54}
    \lVert (\Pt-\Id) w\rVert_{L^2_e(0,T)} \apprle h_t \lVert \partial_t(\Pt-\Id)w \rVert_{L^2_e(0,T)} \quad \text{for all~}  w\in H_{0,\bullet}^1(0,T),
\end{equation}
where we used the notation~$\Pt$ to also denote the elliptic projector acting on functions in $H^1(0,T)$. This derivation is standard and we omit it. Then, the result follows by inserting estimate~\eqref{eq:54} within~\eqref{eq:53}, and using Assumption~\ref{ass:58}.
\end{proof}
\noindent
The following lemma establishes the approximation properties of compositions of the projectors needed to prove explicit convergence estimates.
\begin{lemma} \label{lem:511}
Let $1 \le s_{\bx} \le p_{\bx}$ and $1 \le s_t \le p_t$, with $p_{\bx}$ and $p_t$ as in Assumptions \ref{ass:57} and \ref{ass:58}, respectively. Then,
\begin{enumerate}[leftmargin=2em, label=(\roman*)]
\item for any $V \in H^1(0,T;H_0^{s_{\bx} +1}(\Omega))$,
\begin{equation*} 
    \| (\Pg-\Id) \partial_t V \|_{{L^2_e(Q_T)}} \apprle h_{\bx}^{s_{\bx}}\| \partial_t V \|_{L^2(0,T;H^{s_{\bx}+1}(\Omega))};
\end{equation*}
%
\item for any $U$ such that $\nabla_{\bx} \cdot (c^2 \nabla_{\bx} U) \in H_{0,\bullet}^{s_t+1}(0,T;L^2(\Omega))$,
\begin{equation*} 
    \| (\Pt - \Id) \nabla_{\bx} \cdot (c^2 \nabla_{\bx} U) \|_{L^2_e(Q_T)} \apprle h_t^{s_t+1}\| \partial_t^{s_t+1} \nabla_{\bx} \cdot (c^2 \nabla_{\bx} U) \|_{L^2(Q_T)};
\end{equation*}
%
\item for any $V \in H_{0,\bullet}^{s_t + 1}(0,T;H^1_0(\Omega))$,
\begin{equation*} 
    \| c \nabla_{\bx} \Pg (\Pt - \Id) V \|_{L^2_e(Q_T)} \apprle h_t^{s_t+1}\lVert \partial_t^{s_t+1} V\rVert_{L^2(0,T;H^1(\Omega))};
\end{equation*}
%
\item for any $U \in H_{0,\bullet}^{s_t+1}(0,T;L^2(\Omega)) \cap H^1(0,T;H_0^{s_{\bx}+1}(\Omega))$,
\begin{align*}
        \| \partial_t (\Pi^{\nabla}_{h_{\bx}} \Pt - \Id)U\rVert_{{L^2_e(Q_T)}}  \apprle h_t^{s_t}\lVert \partial_t^{s_t+1} U\rVert_{L^2(Q_T)} + h^{s_{\bx}}_{\bx} \lVert \partial_t U \rVert_{L^2(0,T;H^{s_{\bx} + 1}(\Omega))};
\end{align*}
%
\item for any $V \in H_{0,\bullet}^{s_t+1}(0,T;H_0^1(\Omega)) \cap L^2(0,T;H_0^{s_{\bx}+1}(\Omega))$,
\begin{align*} 
    \lVert (\Pg\Pt - \Id) V\rVert_{{L^2_e(Q_T)}} \apprle & h_t^{s_t+1}\lVert \partial_t^{s_t+1} V \rVert_{L^2(0,T;H^1(\Omega))} +  h_{\bx}^{s_{\bx}}\lVert V \rVert_{L^2(0,T;H^{s_{\bx} + 1}(\Omega))};
\end{align*}
%
\item for any $U \in H_{0,\bullet}^{s_t+1}(0,T;H^1(\Omega)) \cap L^2(0,T;H_0^{s_{\bx+1}}(\Omega))$,
\begin{align*}
    \lVert c\nabla_{\bx} (\Pg\Pt - \Id)U\rVert_{{L^2_e(Q_T)}} \apprle & h^{s_t+1}_t \lVert \partial_t^{s_t+1} U \rVert_{L^2(0,T;H^1(\Omega))} +  h_{\bx}^{s_{\bx}}\lVert U \rVert_{L^2(0,T;H^{s_{\bx} + 1}(\Omega))};
\end{align*}
%
\item for any $V \in H_{0,\bullet}^{s_t + 1}(0,T;L^2(\Omega)) \cap H^1(0,T;H_0^{s_{\bx} +1}(\Omega))$,
\begin{align*}
    \lVert \partial_t (\Pg \Pt - \Id) V \rVert_{\NOmega_{e,{\bh}}} & \apprle h_t^{s_t}\lVert \partial_t^{s_t+1} V\rVert_{L^2(Q_T)} + h_{\bx}^{s_{\bx}}\lVert \partial_t V \rVert_{L^2(0,T;H^{s_{\bx}+1}(\Omega))};
\end{align*}
%
\item for any 
\begin{align*}
    & U \in H_{0,\bullet}^{s_t+1}(0,T;H^1(\Omega)) \cap H^1(0,T;H_0^{s_{\bx}+1}(\Omega)),
    \\ & V \in H_{0,\bullet}^{s_t+1}(0,T;L^2(\Omega)) \cap H^1(0,T;H_0^{s_{\bx}+1}(\Omega)),
\end{align*}
it holds true that
\begin{equation*}
\begin{aligned} 
    \|((\Pg\Pt-\Id) U,(&\Pg\Pt-\Id) V) \|_{\mathcal{V}_{e,{\bh}}(Q_T)} \\
     &\apprle h_t^{s_t}\lVert \partial_t^{s_t+1} U\rVert_{L^2(0,T;H^1(\Omega))}+ h^{s_{\bx}}_{\bx} \lVert \partial_t U \rVert_{L^2(0,T;H^{s_{\bx} + 1}(\Omega))}
    \\ &\quad+ h_t^{s_t}\lVert \partial_t^{s_t+1} V\rVert_{L^2(Q_T)} + h_{\bx}^{s_{\bx}}\lVert \partial_t V \rVert_{L^2(0,T;H^{s_{\bx+1}}(\Omega))}.
\end{aligned}
\end{equation*}
\end{enumerate}
\end{lemma}
\begin{proof}
For the estimate in~\emph{(i)}, we apply the Poincar\'e inequality in space 
\begin{equation*}
     \lVert (\Pi^{\nabla}_{h_x} - \Id)\partial_tV \rVert_{L^2_e(Q_T)} \apprle \lVert  c \nabla_x(\Pi^{\nabla}_{h_x} - \Id)\partial_tV\rVert_{L^2_e(Q_T)},
\end{equation*}
the approximation property in~\eqref{eq:43}, and Assumption~\ref{ass:57}:
\begin{equation*}
    \begin{aligned}
        &\lVert c \nabla_x(\Pi^{\nabla}_{h_x} - \Id)\partial_tV\rVert_{L_e^2(Q_T)}  
        \apprle \inf_{W_{h_{\bx}} \in S_{h_{\bx}}(\Omega) \otimes L^2(0,T)} \| c\nabla_{\bx} (W_{h_{\bx}} - \partial_t V) \|_{L_e^2(Q_T)}
        \\ &\qquad \apprle h_{\bx}^{s_{\bx}} \left(\int_0^T \| \partial_t V(\cdot,s) \|^2_{H^{s_{\bx}+1}(\Omega)} e^{-s/T} \, \dd s\right)^{\frac{1}{2}}
        \apprle h_{\bx}^{s_{\bx}} \| \partial_t V \|_{L^2(0,T;H^{s_{\bx}+1}(\Omega))}.
    \end{aligned}
\end{equation*}

\noindent 
To obtain \emph{(ii)}, we simply use Lemma \ref{lem:510}, while~\emph{(iii)} is obtained by combining the stability property in~\eqref{eq:45} with Lemma \ref{lem:510}.

\noindent
In the subsequent steps of the proof, we repeatedly use the following identities, which follow from Lemma~\ref{lem:55}:
\begin{align*}
    \Pg \Pt - \Id = \Pt\Pg - \Id & = \Pt - \Id + \Pt(\Pi^{\nabla}_{h_{\bx}} - \Id) 
    \\ & = \Pi^{\nabla}_{h_{\bx}}(\Pt - \Id) + \Pi^{\nabla}_{h_{\bx}} - \Id.
\end{align*}
%

\noindent
To prove the estimate in~\emph{(iv)}, we compute
\begin{align*}
    \lVert \partial_t(\Pg \Pt - \Id)U\rVert_{{L^2_e(Q_T)}} & \leq \lVert \partial_t (\Pt - \Id)U \rVert_{{L^2_e(Q_T)}} + \lVert\partial_t \Pt (\Pg - \Id)U \rVert_{{L^2_e(Q_T)}} \\
    & \leq \lVert\partial_t (\Pt - \Id) U \rVert_{{L^2_e(Q_T)}} + \lVert (\Pg - \Id) \partial_t U \rVert_{{L^2_e(Q_T)}},
    \\ & \apprle \lVert\partial_t (\Pt - \Id) U \rVert_{{L^2_e(Q_T)}} + \lVert c \nabla_{\bx} (\Pg - \Id) \partial_t U \rVert_{{L^2_e(Q_T)}},
\end{align*}
where we also used the stability property in \eqref{eq:46}, the Poincar\'e inequality in space and the commutativity properties in Lemma \ref{lem:55}. Then,~\emph{(iv)} follows from the approximation properties~\eqref{eq:43} and~\eqref{eq:44}, combined with Assumptions~\ref{ass:57} and~\ref{ass:58}.

\noindent
For the inequality in~\emph{(v)}, the Poincar\'e inequality in space, the stability property in~\eqref{eq:45}, and the commutativity properties in Lemma \ref{lem:55} imply
\begin{align*}
    \lVert (\Pg\Pt - \Id) V \rVert_{{L^2_e(Q_T)}} & \leq \lVert \Pg(\Pt - \Id) V\rVert_{{L^2_e(Q_T)}} + \lVert (\Pg-\Id) V\rVert_{{L^2_e(Q_T)}} 
    \\ & \apprle \lVert \nabla_{\bx} \Pg(\Pt - \Id) V\rVert_{{L^2_e(Q_T)}} + \lVert \nabla_{\bx} (\Pg-\Id) V\rVert_{{L^2_e(Q_T)}} 
    \\ & \apprle  \lVert (\Pt - \Id) \nabla_{\bx} V\rVert_{{L^2_e(Q_T)}} + \lVert \nabla_{\bx} (\Pg - \Id) V\rVert_{{L^2_e(Q_T)}}.
\end{align*}
We conclude with Lemma \ref{lem:510} and the approximation property \eqref{eq:43} of $\Pg$, combined with Assumption \ref{ass:57}.

\noindent
For~\emph{(vi)}, we proceed similarly, and compute, with the stability property in~\eqref{eq:45} and the commutativity properties in Lemma \ref{lem:55}
\begin{align*}
     \lVert c \nabla_{\bx} (\Pg \Pt - \Id)U\rVert_{{L^2_e(Q_T)}} &
    \\ & \hspace{-1cm} \leq \lVert c \nabla_{\bx} \Pg (\Pt - \Id)U \rVert_{{L^2_e(Q_T)}} + \lVert c \nabla_{\bx} (\Pg - \Id)U \rVert_{{L^2_e(Q_T)}}
    \\ & \hspace{-1cm}
    \leq \lVert (\Pt - \Id)  c\nabla_{\bx} U \rVert_{{L^2_e(Q_T)}}+ \lVert c \nabla_{\bx} (\Pg - \Id)U \rVert_{{L^2_e(Q_T)}},
\end{align*}
and conclude exactly as for~\emph{(v)}.

\noindent
For~\emph{(vii)}, we use Lemma \ref{lem:51} to first obtain
\begin{align*}
    \lVert \partial_t (\Pg \Pt - \Id) V\rVert_{\NOmega_{e,{\bh}}} &\apprle \lVert \partial_t (\Pg \Pt - \Id) V\rVert_{{L^2_e(Q_T)}},
\end{align*}
and then proceed as in~\emph{(iv)}.

\noindent
Finally, the projection error estimate in the~$\mathcal{V}_{e,\bh}(Q_T)$ norm stated in~\emph{(viii)} readily follows from~\emph{(iv)}--\emph{(vii)} and the Poincar\'e inequality applied in space and in time. 
\end{proof}
\begin{remark}
Under a standard elliptic regularity assumption in space, improved estimates for the projection error associated with~$\Pg$ in the~$L^2_e(Q_T)$ norm can be derived using a duality argument. We emphasize that this assumption is not needed for the error estimate in Theorem~\ref{th:514}. However, in Section~\ref{sec:53} below, we rely on it, along with the resulting improved projection error estimates, to derive error bounds in the $L^2$ norm.
\eremk \end{remark}
\begin{remark}\label{rem:513}
When~$U$ and~$V$ are sufficiently smooth, we deduce from Lemma~\ref{lem:56} and~\emph{(i)}--\emph{(iii)} in Lemma~\ref{lem:511} that
\begin{align*}
    \| (\Pg \Pt U- U_{\bh}, \Pg \Pt V - V_{\bh})\|_{\mathcal{V}_{e,\bh}(Q_T)} \apprle h_{\bx}^{s_{\bx}}+h_t^{s_t+1},
\end{align*}
while, for the projection error, the best achievable estimate in the norm~$\| \cdot \|_{\mathcal{V}_{e,\bh}(Q_T)}$ is
\begin{equation*}
    \|((\Pg\Pt-\Id) U,(\Pg\Pt-\Id) V)\|_{\mathcal{V}_{e,\bh}(Q_T)}\apprle h_{\bx}^{s_{\bx}}+h_t^{s_t}.
\end{equation*}
\eremk \end{remark}
\noindent
Combining Lemmas \ref{lem:56} and \ref{lem:511} yields the main result.
\begin{theorem} \label{th:514}
Let us assume the regularity on the data as in Assumption~\ref{ass:21}. Let~$(U,V)$ be the unique solution to~\eqref{eq:36}, and~$(U_{\bh},V_{\bh})$ be the unique discrete solution to~\eqref{eq:37}. Let $1 \le s_{\bx} \le p_{\bx}$ and $1 \le s_t \le p_t$, with $p_{\bx}$ and $p_t$ as in Assumptions \ref{ass:57} and \ref{ass:58}, respectively. Assume that~$(U,V)$ satisfies 
\begin{align*}
    & U \in H_{0,\bullet}^{s_t+1}(0,T;H^1(\Omega)) \cap H^1(0,T;H_0^{s_{\bx}+1}(\Omega)),
    \\ & V \in H_{0,\bullet}^{s_t+1}(0,T;H^1(\Omega)) \cap H^1(0,T;H_0^{s_{\bx}+1}(\Omega)),
    \\ &  \nabla_{\bx} \cdot (c^2 \nabla_{\bx} U) \in H_{0,\bullet}^{s_t+1}(0,T;L^2(\Omega)).
\end{align*}
Then, we have the following error estimate:
\begin{align*}
    \| (U,V) - (U_{\bh},V_{\bh})\|_{\mathcal{V}_{e,\bh}(Q_T)} \apprle & \,\,\, 
    h^{s_t}_t \lVert \partial_t^{s_t+1} U \rVert_{L^2(0,T;H^1(\Omega))} + h^{s_{\bx}}_{\bx} \lVert \partial_t U \rVert_{L^2(0,T;H^{s_{\bx+1}}(\Omega))} 
    \\  + & h_t^{s_t}\lVert \partial_t^{s_t+1} V\rVert_{L^2(0,T;H^1(\Omega))} + h_{\bx}^{s_{\bx}}\lVert \partial_t V \rVert_{L^2(0,T;H^{s_{\bx}+1}(\Omega))}
    \\  + & h_t^{s_t+1} \| \partial_t^{s_t+1} \nabla_{\bx} \cdot (c^2 \nabla_{\bx} U) \|_{L^2(Q_T)}.
\end{align*}
Moreover, for the term~$\|c \nabla_{\bx} (U - U_{\bh}) \|_{L^2_e(Q_T)}$ alone, we have
\begin{align*}
    \|c \nabla_{\bx} (U - U_{\bh}) \|_{L^2_e(Q_T)} & 
     \\ & \hspace{-3.5cm} \apprle   h^{s_t+1}_t \lVert \partial_t^{s_t+1} U \rVert_{L^2(0,T;H^1(\Omega))} + h_{\bx}^{s_{\bx}}\lVert U \rVert_{L^2(0,T;H^{s_{\bx} + 1}(\Omega))} +h_{\bx}^{s_{\bx}}\lVert \partial_t V \rVert_{L^2(0,T;H^{s_{\bx} + 1}(\Omega))}
     \\ + & h_t^{s_t+1} (\| \partial_t^{s_t+1} \nabla_{\bx} \cdot (c^2 \nabla_{\bx} U) \|_{L^2(Q_T)} + \lVert \partial_t^{s_t+1} V\rVert_{L^2(0,T;H^1(\Omega))}).
\end{align*}
\end{theorem}
\begin{proof}
For all $(W_{\bh},Z_{\bh}) \in (Q_{\bh}(Q_T))^2$, with the triangle inequality, we deduce 
\begin{equation} \label{eq:58}
\begin{aligned}
    \| (U,V) - (U_{\bh},V_{\bh})\|_{\mathcal{V}_{e,\bh}(Q_T)} \le & \| (U,V) - (W_{\bh},Z_{\bh})\|_{\mathcal{V}_{e,\bh}(Q_T)} 
    \\ & + \| (W_{\bh},Z_{\bh}) - (U_{\bh},V_{\bh})\|_{\mathcal{V}_{e,\bh}(Q_T)}.
\end{aligned}
\end{equation}
We choose $W_{\bh} := \Pg \Pt U$ and $Z_{\bh} := \Pg \Pt V$\, where the projectors are defined in~\eqref{eq:41} and \eqref{eq:42}. Then, we obtain the estimate in the $\|\cdot\|_{\mathcal{V}_{e,\bh}(Q_T)}$ norm by using~\emph{(viii)} in Lemma~\ref{lem:511} to estimate~$\| (U,V) - (W_{\bh},Z_{\bh})\|_{\mathcal{V}_{e,\bh}(Q_T)}$, and combining Lemma~\ref{lem:56} with \emph{(i)}--\emph{(iii)} in Lemma~\ref{lem:511} to estimate~$\| (W_{\bh},Z_{\bh}) - (U_{\bh},V_{\bh})\|_{\mathcal{V}_{e,\bh}(Q_T)}$. \\ 
\noindent
For the second estimate, the triangle inequality, together with the definition \eqref{eq:40} of $\| \cdot \|_{\mathcal{V}_{e,\bh}(Q_T)}$, gives
\begin{align}
    \|c \nabla_{\bx} (U - U_{\bh}) \|_{L^2_e(Q_T)} & \le  \|c \nabla_{\bx} (W_{\bh} - U_{\bh}) \|_{L^2_e(Q_T)} + \|c \nabla_{\bx} (W_{\bh}-U) \|_{L^2_e(Q_T)}
    \nonumber
    \\ & \le \| (W_{\bh},Z_{\bh}) - (U_{\bh},V_{\bh})\|_{\mathcal{V}_{e,\bh}(Q_T)}
    + \|c \nabla_{\bx} (W_{\bh}-U) \|_{L^2_e(Q_T)}.
    \label{eq:59}
\end{align}
We choose again~$W_{\bh} := \Pg \Pt U$ and $Z_{\bh} := \Pg \Pt V$, and conclude estimating $\| (W_{\bh},Z_{\bh})-(U_{\bh},V_{\bh})\|_{\mathcal{V}_{e,\bh}(Q_T)}$ as in the first part, and~$\|c \nabla_{\bx} (W_{\bh}-U) \|_{L^2_e(Q_T)}$ by using~\emph{(vi)} in Lemma~\ref{lem:511}.
\end{proof}
\begin{remark}
As pointed out in Remark~\ref{rem:513}, the projection error term dominates in~\eqref{eq:58}, yielding 
\begin{equation*}
    \| (U,V) - (U_{\bh},V_{\bh})\|_{\mathcal{V}_{e,\bh}(Q_T)} \apprle h_{\bx}^{s_{\bx}}+h_t^{s_t}.
\end{equation*}
On the contrary, both terms on the right-hand side of~\eqref{eq:59} admit similar bounds, allowing to obtain
\begin{equation*}
    \|c \nabla_{\bx} (U - U_{\bh}) \|_{L^2_e(Q_T)} \apprle h_{\bx}^{s_{\bx}}+h_t^{s_t+1}.
\end{equation*}
\end{remark}

\subsubsection{Error estimates in the \texorpdfstring{$L^2$}{L2} norm} \label{sec:53}

\noindent
In this section, we derive improved error estimates for~$\| U - U_{\bh}\|_{L_e^2(Q_T)}$ and $\|V- V_{\bh}\|_{L_e^2(Q_T)}$ under the following standard elliptic regularity assumption in space.
\begin{assumption}\label{ass:515}(elliptic regularity) We assume that, for any~$g \in L^2(\Omega)$, the unique solution to the problem: find~$\phi \in H_0^1(\Omega)$ such that
\begin{equation*}
    (c^2 \nabla_{\bx} \phi, \nabla_{\bx} \psi)_{L^2(\Omega)} = (g,\psi)_{L^2(\Omega) \quad} \quad \text{for all~} \psi \in H^1_0(\Omega)
\end{equation*}
belongs to $H^2(\Omega)$ and satisfies the stability estimate $\| \phi \|_{H^2(\Omega)} \apprle \| g \|_{L^2(\Omega)}$.
\end{assumption}
\begin{remark}
Assumption \ref{ass:515} is satisfied when $c^2$ is a Lipschitz function in~$\Omega$, and $\Omega$ is convex and Lipschitz, or $\partial\Omega$ is of class~$C^{1,1}$. 
\eremk \end{remark}

\noindent
Under Assumptions~\ref{ass:57} and~\ref{ass:515}, a standard duality argument yields
\begin{equation} \label{eq:60}
    \|(\Pg-\Id)W\|_{L^2_e(Q_T)}\apprle h_{\bx}\|\nabla_{\bx}(\Pg-\Id)W\|_{L^2_e(Q_T)} \quad \forall \ W\in L^2(0,T;H^1_0(\Omega)),
\end{equation}
which leads to error estimates in the~$L^2_e(Q_T)$ norm with improved converge rates in~$h_{\bx}$. To achieve this, we first refine the projection error estimates~\emph{(i)} and~\emph{(v)} of Lemma~\ref{lem:511}.
\begin{lemma} \label{lem:517}
Let $1 \le s_{\bx} \le p_{\bx}$ and $1 \le s_t \le p_t$ , with $p_{\bx}$ and $p_{t}$ as in Assumptions \ref{ass:57} and \ref{ass:58}, respectively. Then, if also Assumption~\ref{ass:515} is satisfied,
we obtain
\begin{enumerate}[leftmargin=2em, label=(\roman*)]
\item for any $W \in H^1(0,T;H_0^{s_{\bx} +1}(\Omega))$,
\begin{equation*} 
    \| (\Pg-\Id) \partial_t W \|_{{L^2_e(Q_T)}} \apprle h_{\bx}^{s_{\bx}+1}\| \partial_t W \|_{L^2(0,T;H^{s_{\bx}+1}(\Omega))};
\end{equation*}
%
\item for any $W \in H_{0,\bullet}^{s_t+1}(0,T;H_0^1(\Omega)) \cap L^2(0,T;H_0^{s_{\bx}+1}(\Omega))$,
\begin{align*} 
    \lVert (\Pg\Pt - \Id) W\rVert_{{L^2_e(Q_T)}} \apprle & h_t^{s_t+1}\lVert \partial_t^{s_t+1} W \rVert_{L^2(0,T;H^1(\Omega))} 
     +  h_{\bx}^{s_{\bx}+1}\lVert W \rVert_{L^2(0,T;H^{s_{\bx} + 1}(\Omega))}.
\end{align*}
\end{enumerate}
\end{lemma}
\begin{proof}
The proofs of \emph{(i)} and \emph{(ii)} proceed analogously to those of \textit{(i)} and \textit{(v)} in Lemma~\ref{lem:511}, respectively, except that, instead of the Poincar\'e inequality in space, we use~\eqref{eq:60} to estimate terms of the form~$\| (\Pi_{h_{\bx}}^\nabla - \Id) \cdot \|_{L^2_e(Q_T)}$, which gives the extra power of~$h_{\bx}$.
\end{proof}
\noindent
We now prove the main result of this section. 
\begin{theorem} \label{th:517}
In addition to the assumptions of Theorem \ref{th:514}, suppose that Assumption~\ref{ass:515} is satisfied. Then, we have 
\begin{equation} \label{eq:61}
\begin{aligned}
     \| U - U_{\bh}\|_{L_e^2(Q_T)} \apprle & \,\,\, 
    h^{s_t+1}_t \lVert \partial_t^{s_t+1} U \rVert_{L^2(0,T;H^1(\Omega))} + h^{s_{\bx}+1}_{\bx} \lVert U \rVert_{L^2(0,T;H^{s_{\bx+1}}(\Omega))} 
     \\ &  + h_t^{s_t+1} (\| \partial_t^{s_t+1} \nabla_{\bx} \cdot (c^2 \nabla_{\bx} U) \|_{L^2(Q_T)} + \lVert \partial_t^{s_t+1} V\rVert_{L^2(0,T;H^1(\Omega))}),
     \\ & + h_{\bx}^{s_{\bx}+1} \| \partial_t V \|_{L^2(0,T;H^{s_{\bx}+1}(\Omega))},
     \end{aligned}
     \end{equation}
\begin{equation} \label{eq:62}
\begin{aligned}
     \| V - V_{\bh}\|_{L_e^2(Q_T)} \apprle & \,\,\, 
    h^{s_t+1}_t \lVert \partial_t^{s_t+1} V \rVert_{L^2(0,T;H^1(\Omega))} + h^{s_{\bx}+1}_{\bx} \lVert \partial_t V \rVert_{L^2(0,T;H^{s_{\bx+1}}(\Omega))} 
    \\ &  + h_t^{s_t+1} (\| \partial_t^{s_t+1} \nabla_{\bx} \cdot (c^2 \nabla_{\bx} U) \|_{L^2(Q_T)} + \lVert \partial_t^{s_t+1} V\rVert_{L^2(0,T;H^1(\Omega))}).
    \end{aligned}
\end{equation}
\end{theorem}
\begin{proof}
Using the triangle inequality, the Poincar\'e inequality in time, and the definition~\eqref{eq:40} of $\| \cdot \|_{\mathcal{V}_{e,\bh}(Q_T)}$, we obtain
\begin{align*}
    \| U - U_{\bh}\|_{L_e^2(Q_T)} & \le \| \Pg \Pt U - U_{\bh} \|_{L_e^2(Q_T)} + \| (\Pg \Pt - \Id) U\|_{L_e^2(Q_T)}
    \\ & \hspace{-2cm} \le \| (\Pg \Pt U- U_{\bh}, \Pg \Pt V - V_{\bh})\|_{\mathcal{V}_{e,\bh}(Q_T)} +  \| (\Pg \Pt - \Id) U\|_{L_e^2(Q_T)}.
\end{align*}
We derive~\eqref{eq:61} combining Lemma \ref{lem:56}, \emph{(i)}, \emph{(ii)} of Lemma \ref{lem:517}, and \emph{(ii)}, \emph{(iii)} of Lemma~\ref{lem:511}. Inequality \eqref{eq:62} is obtained similarly, except that the Poincar\'e inequality is not needed.
\end{proof}
\subsection{Numerical results}

In this section, we present two numerical tests in one space dimension to validate the stability and convergence of the proposed method.\footnote{All numerical test are performed with Matlab R2024a. The codes used for the numerical tests are available in the GitHub repository \cite{FerrariCodes2025}.}
The first test involves a smooth solution. We demonstrate unconditional stability and observe quasi-optimal convergence rates using standard space--time tensor product discretizations. In the second test, we consider a singular solution and still observe convergence, even though this case is not covered by our theory.

\noindent
An important aspect not addressed in this work is the efficient implementation of the space--time method, including its potential reformulation as a time-marching scheme; we refer to \cite[\S 5.1]{FerrariFraschiniLoliPerugia2024} for a detailed discussion.

\subsubsection{Smooth solution}

As a first numerical test, we consider problem \eqref{eq:1} in the domain $Q_T = \Omega \times (0,3)$, with~$\Omega = (0,1)$, $c^2(x) = x+1$, $U_0(x) = \sin(\pi x)$, $V_0(x)=0$, and $F(x,t)$ such that the exact solution is
\begin{equation} \label{eq:66}
    U(x,t) = \sin^2 \left( \frac{5}{4} \pi t \right) \sin(\pi x) + \sin(\pi x). \vspace{-0.2cm}
\end{equation}
\paragraph{Stability} We first demonstrate stability of the scheme in \eqref{eq:37} with respect to the norm $\| \cdot \|_{\mathcal{V}_{e,\bh}(Q_T)}$, as defined in \eqref{eq:40}, in line with the result of Corollary~\ref{cor:53}. The relative error in this norm is computed using first splines with maximal regularity and then $C^1$ splines in both time and space, all with the same polynomial degree $p$ but different mesh sizes. Specifically, we fix the time step $h_t$ and progressively refine the spatial mesh size $h_x$. As shown in Figure~\ref{fig:5}, where the results for several values of~$p$ are reported, no oscillations or instabilities are observed, and no CFL-type constraint of the form $h_t \le C h_x$ appears to be required. Here and below, we do not test with~$C^0$ polynomials in time as, in this case, our scheme does not differ significantly from that in~\cite{FrenchPeterson1996}.

\begin{figure}[h!]
\centering
\begin{minipage}{0.485\textwidth}
\includegraphics[width=\linewidth]{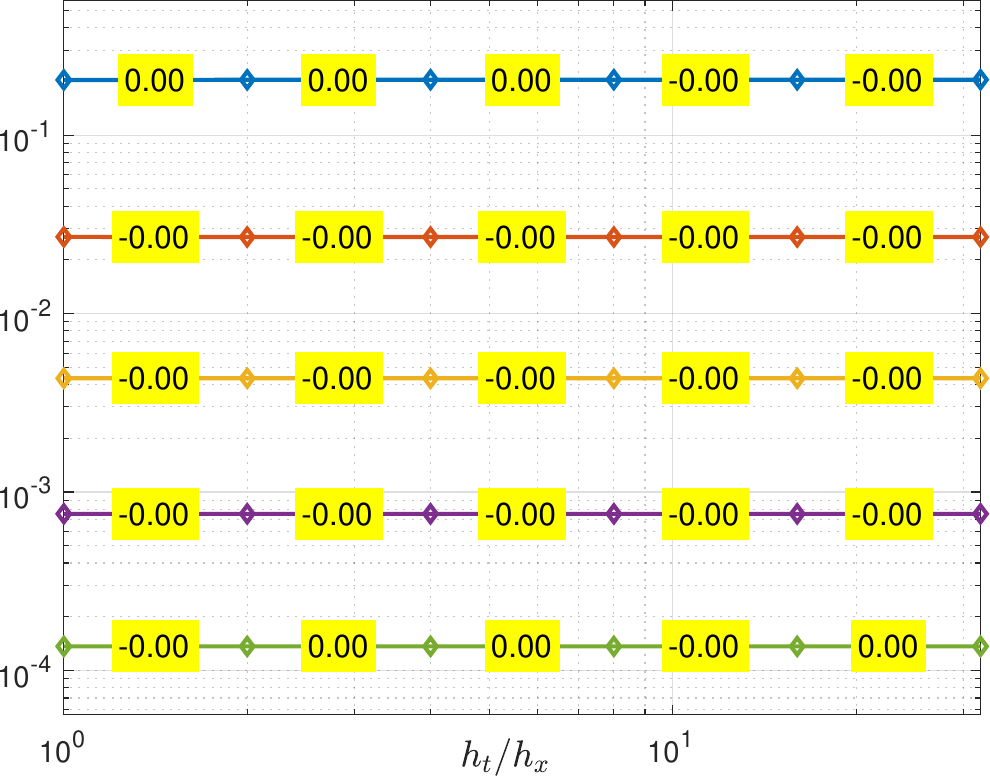}
\end{minipage}
\begin{minipage}{0.485\textwidth}
\includegraphics[width=\linewidth]{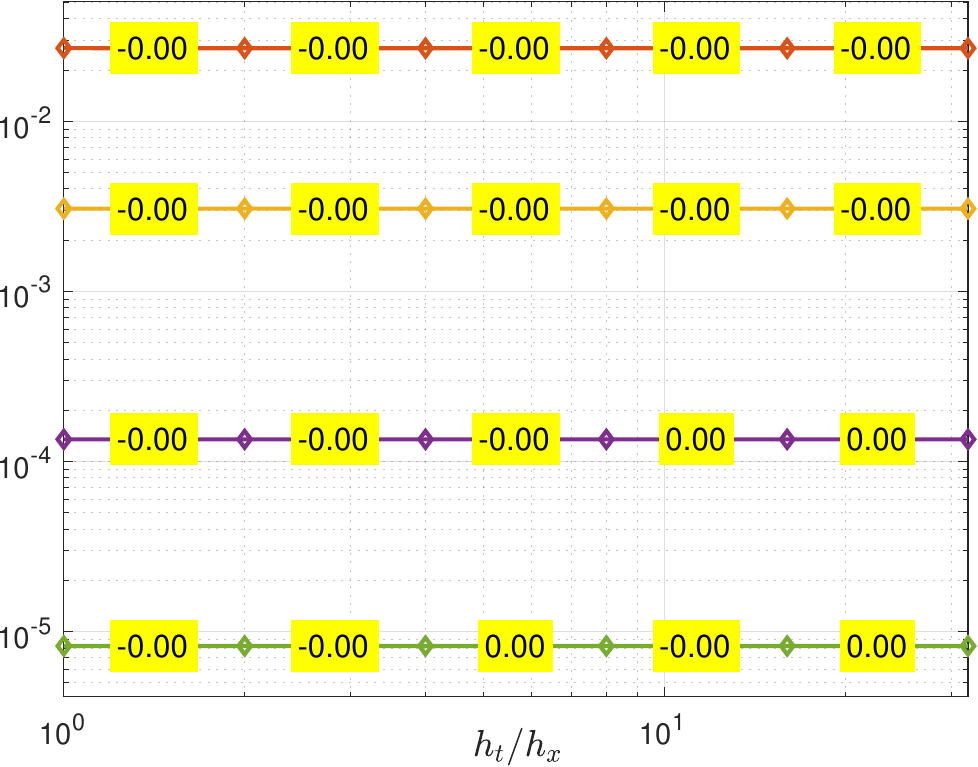}
\end{minipage}
\caption{{\small Test with the smooth solution given by~\eqref{eq:66}. Relative errors in the norm $\mathcal{V}_{e,\bh}(Q_T)$ for the discrete solution computed with the scheme in \eqref{eq:37}. The results are obtained by fixing the temporal mesh size $h_t$, decreasing the spatial mesh size $h_x$ and varying the polynomial degree $p$: \textcolor{myblue}{$p=1$}, \textcolor{myorange}{$p=2$}, \textcolor{myyellow}{$p=3$}, \textcolor{mypurple}{$p=4$}, \textcolor{mygreen}{$p=5$}. The left plot corresponds to maximal regularity splines, while the right plot shows results obtained using $C^1$ splines.}}
\label{fig:5}
\end{figure}

\paragraph{Convergence} To validate the convergence results established in Theorems \ref{th:514} and \ref{th:517}, we compute the relative errors in the $\mathcal{V}_{e,\bh}(Q_T)$ norm for $(U_{\bh}, V_{\bh})$, and in the $L^2(Q_T)$ norm for $U_{\bh}$. When using splines in both space and time with equal polynomial degree $p$ and uniform mesh size $h = h_t = h_x$, Theorem~\ref{th:514} predicts a convergence rate of order $\mathcal{O}(h^p)$ in the $\mathcal{V}_{e,\bh}(Q_T)$ norm. Furthermore, from Theorem~\ref{th:517}, we expect a convergence rate of order $\mathcal{O}(h^{p+1})$ in the $L^2$ norm. These theoretical rates are confirmed numerically in Figures~\ref{fig:3} and \ref{fig:4} for both maximal regularity and $C^1$-splines. As expected, for~$\|V-V_{\bh}\|_{L^2(Q_T)}$, we obtained the same convergence rates as for~$\|U-U_{\bh}\|_{L^2(Q_T)}$. We do not report these results here, for brevity.

\begin{figure}[h!]
\centering
\begin{minipage}{0.485\textwidth}
\includegraphics[width=\linewidth]{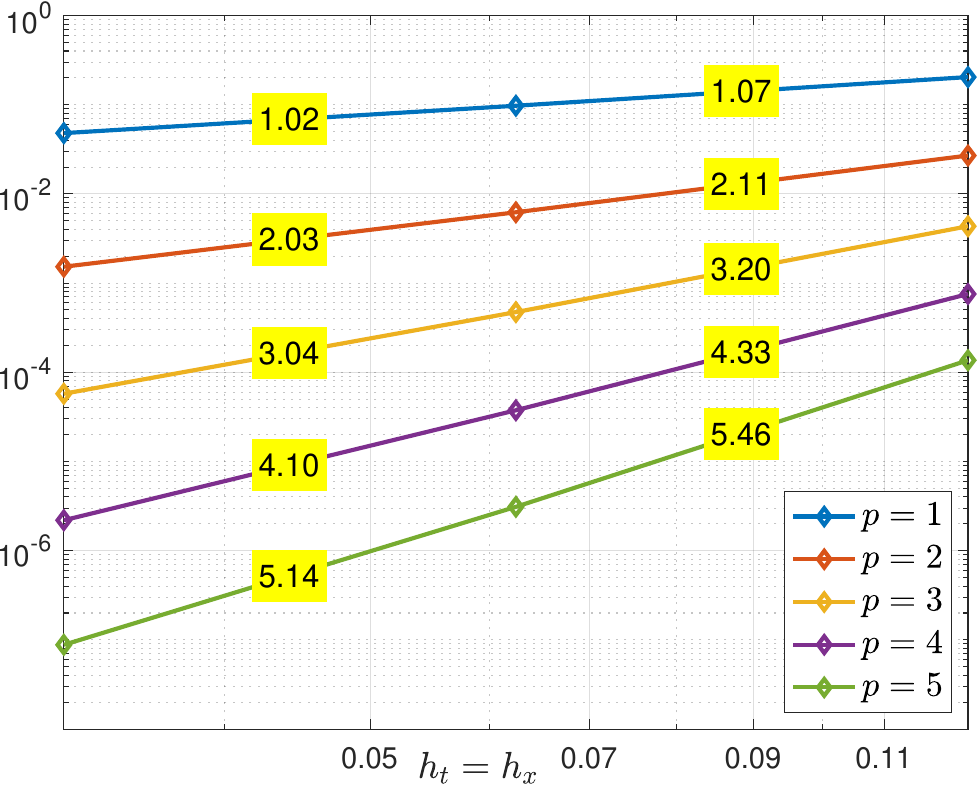}
\end{minipage}
\begin{minipage}{0.485\textwidth}
\includegraphics[width=\linewidth]{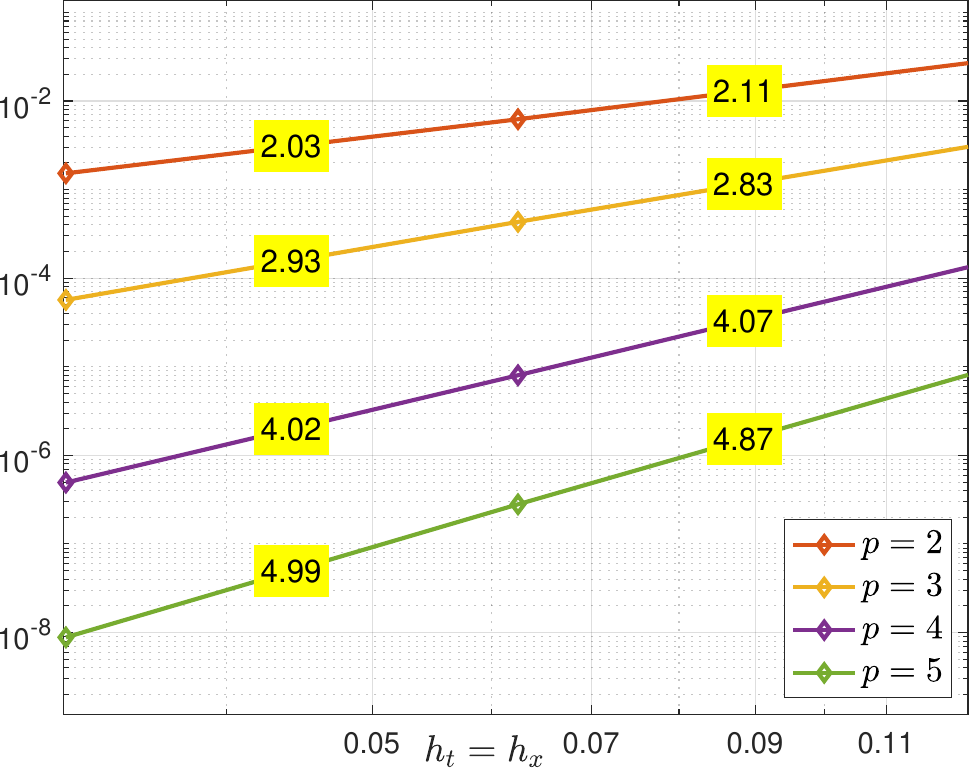}
\end{minipage}
\caption{{\small Test with the smooth solution given by~\eqref{eq:66}. Relative errors in the norm $\mathcal{V}_{e,\bh}(Q_T)$ for the discrete solution computed with the scheme in \eqref{eq:37}. The results are obtained by varying the spatial and temporal mesh sizes  $h_t=h_x$ and varying the polynomial degree~$p$. The left plot corresponds to maximal regularity splines, while the right plot shows results obtained using $C^1$-splines.}}
\label{fig:3}
\end{figure}

\begin{figure}[h!]
\centering
\begin{minipage}{0.485\textwidth}
\includegraphics[width=\linewidth]{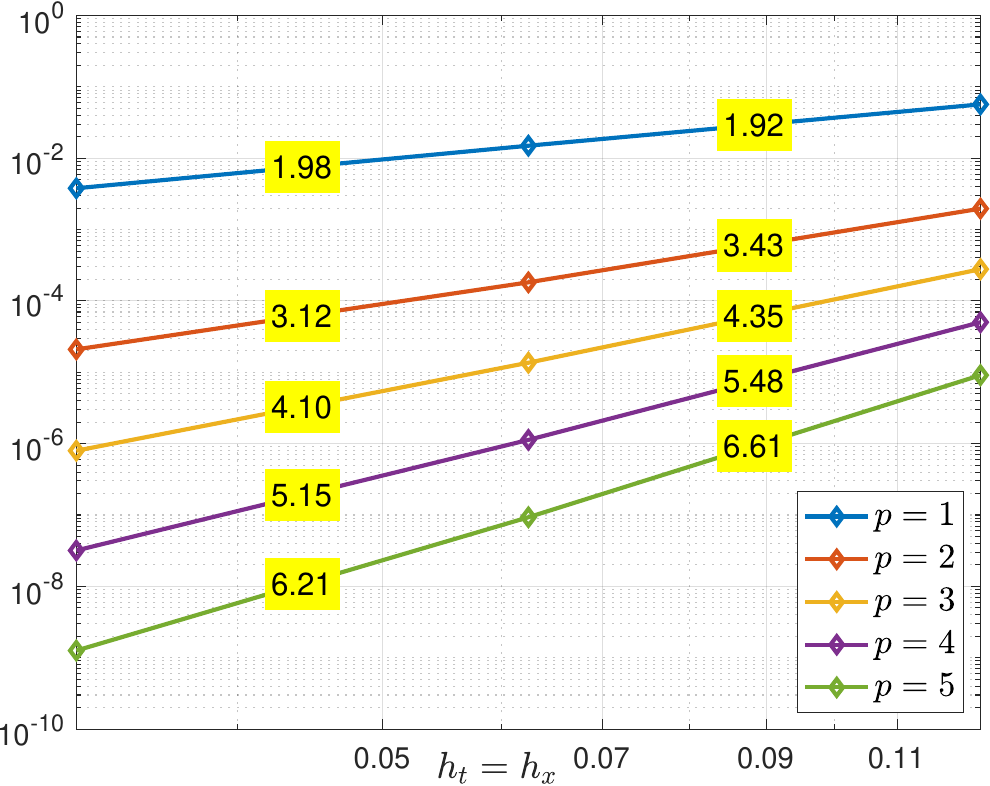}
\end{minipage}
\begin{minipage}{0.485\textwidth}
\includegraphics[width=\linewidth]{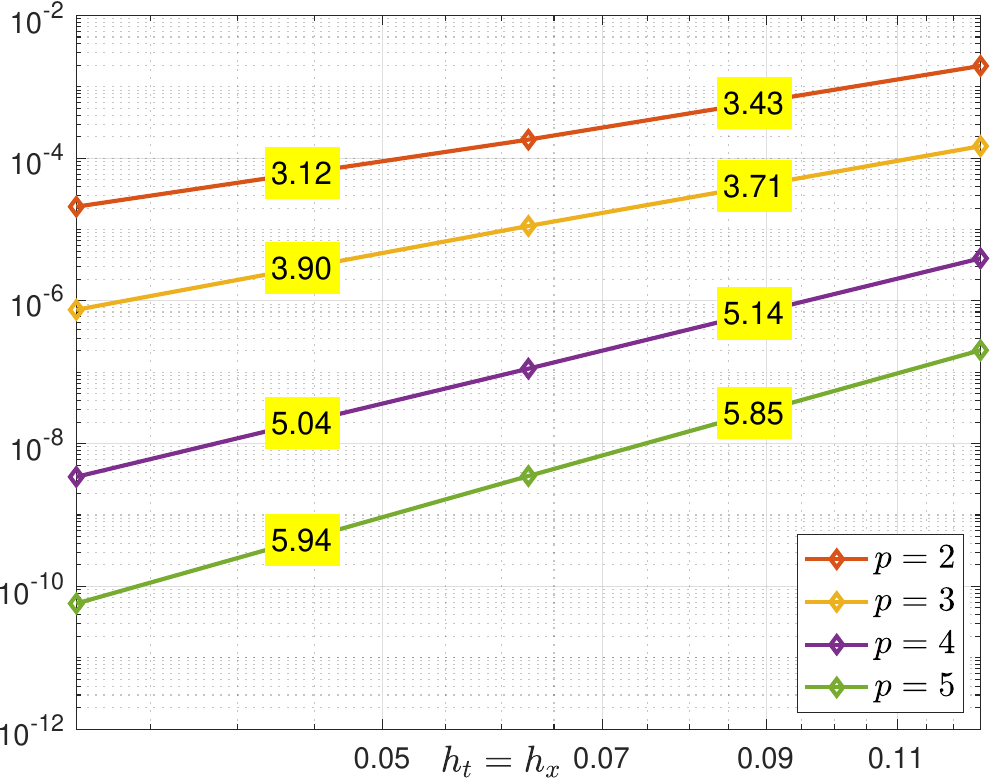}
\end{minipage}
\caption{{\small Test with the smooth solution given by~\eqref{eq:66}. Relative errors in the norm $L^2(Q_T)$ for $U_{\bh}$, the first component of the discrete solution computed with the scheme in \eqref{eq:37}. The results are obtained by varying the spatial and temporal mesh sizes  $h_t=h_x$ and varying the polynomial degree $p$. The left plot corresponds to maximal regularity splines, while the right plot shows results obtained using $C^1$-splines.}}
\label{fig:4}
\end{figure}
 
\subsubsection{Singular solution}
As a second test, we consider the initial value boundary problem \eqref{eq:1} with the non-smooth traveling wave solution of~\cite[\S 7, Problem 3]{BignardiMoiola2023}). The domain is given by $Q_T = \Omega \times (0,1)$, with~$\Omega = (-1.5,1.5)$, and we set $F(x,t)=0$, $c(x)=1$. The initial conditions are chosen so that the exact solution is 
\begin{equation} \label{eq:67}
    U(x,t) = \omega(x-t+1) \mathds{1}_{>0}(x-t+1),
\end{equation}
where $\mathds{1}_{>0}$ is the indicator function of $(0,\infty)$, and $\omega(s) := e^{-20(s-0.1)^2} - e^{-20(s+0.1)^2}$.

\noindent
For $t \in (0,1)$, the solution satisfies $|U(\pm 1.5,t)| \le 10^{-17}$, so we impose homogeneous boundary conditions in practice. The solution $U(x,t)$ is non-smooth along the line $x-~t+1=0$, due to a jump discontinuity in $\partial_t U$ across that line (since $\omega'(0)~=~8e^{-0.2} \ne 0$). Consequently, $U$ lies in~$H^{3/2-\varepsilon}(Q_T)$ for any $\varepsilon >0$, but not in~$H^{3/2}(Q_T)$. We test the convergence of the discrete scheme \eqref{eq:37} using uniform meshes and maximal regularity splines of degree $p$ in both space and time. Although the regularity assumptions on $V=\partial_t U$ required by Theorem \ref{th:514} are not satisfied, we observe quasi-optimal numerical convergence in practice: $\mathcal{O}(h^{3/2})$ for $U$ and $\mathcal{O}(h^{1/2})$ for $V$. These results are presented in Figure~\ref{fig:7}.
\begin{figure}[h!]
\centering
\begin{minipage}{0.485\textwidth}
\includegraphics[width=\linewidth]{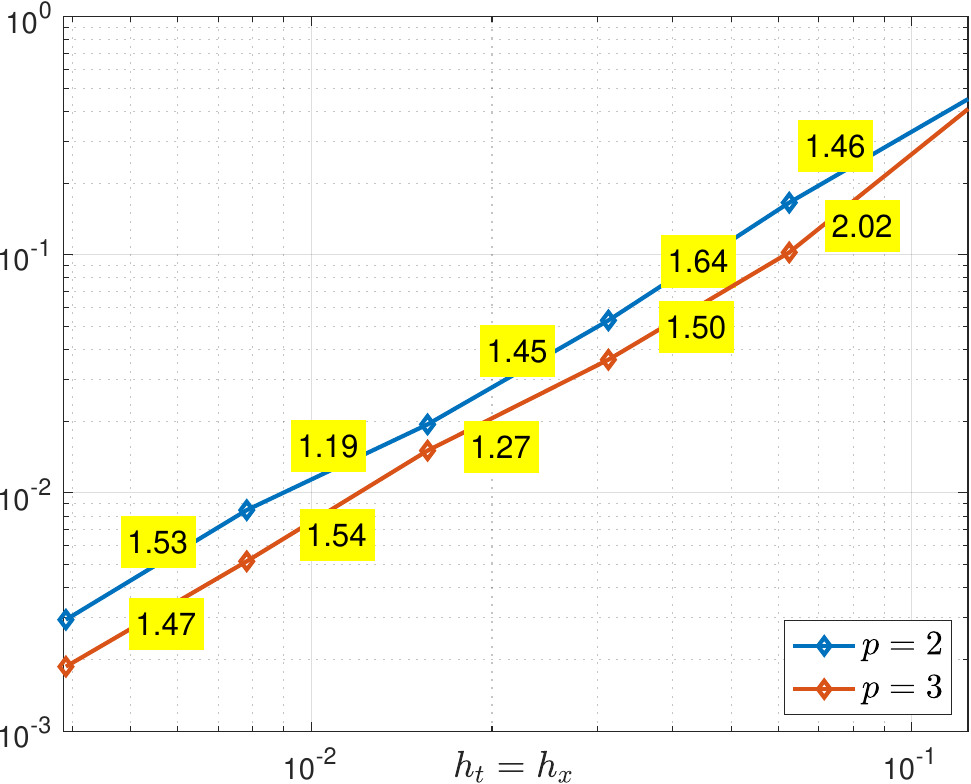}
\end{minipage}
\begin{minipage}{0.485\textwidth}
\includegraphics[width=\linewidth]{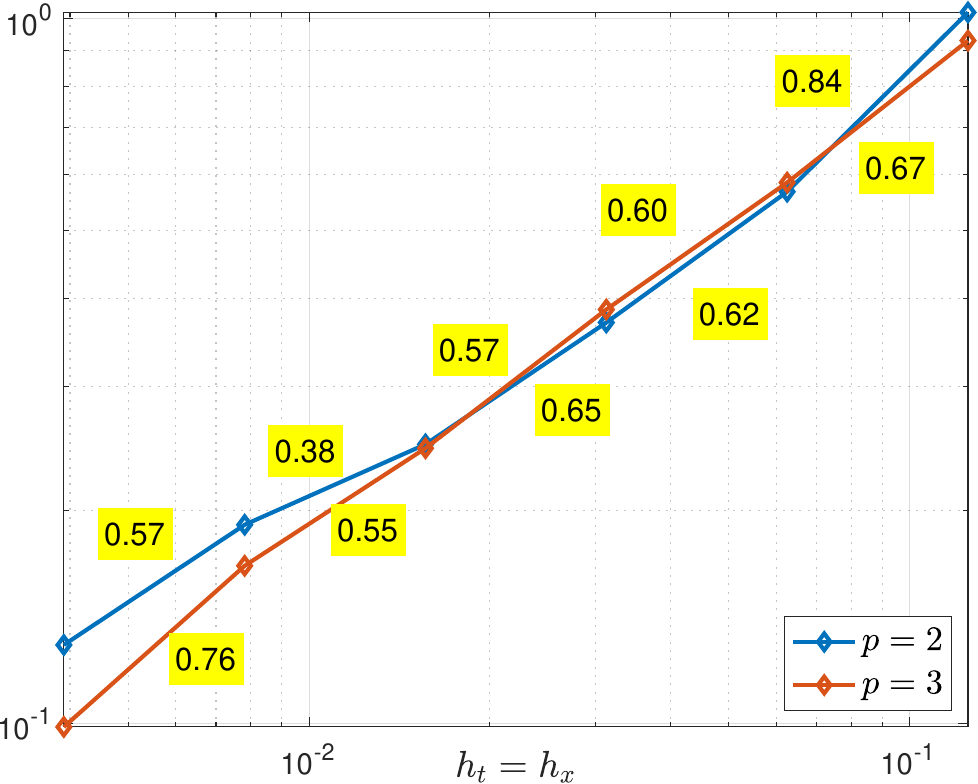}
\end{minipage}
\caption{{\small Test with the singular solution given by~\eqref{eq:67}. Relative errors in the norm $L^2(Q_T)$ for $U_{\bh}$ (left plot) and $V_{\bh}$ (right plot), the two components of the discrete solution with the scheme in \eqref{eq:37}. Maximal regularity splines are employed in both space and time. The results are obtained by varying the spatial and temporal mesh sizes $h_t=h_x$, and for polynomial degrees $p=2,3$.}}
\label{fig:7}
\end{figure}

\section{Conclusion}
In this work, we introduced a conforming space--time tensor product discretization method for the wave equation based on a first-order-in-time variational formulation with exponential weights. The proposed method achieves unconditional stability and quasi-optimal convergence for arbitrary choices of discrete spatial and temporal spaces, under standard approximation assumptions. Numerical tests validate the theoretical results. The key idea is to develop a proof of the inf-sup condition at the continuous level that can be extended directly to the discrete case. This approach allows us to prove well-posedness and convergence with considerable flexibility in the choice of discrete spaces. As in the second-order-in-time formulation of~\cite{FerrariPerugia2025}, exponential weights in time play a fundamental role in the theoretical analysis, particularly in establishing the inf-sup condition. However, from a practical point of view, they may be unnecessary. For instance, for the first-order-in-time formulation, they are proven not to be needed in the case of continuous piecewise polynomial approximations in time \cite{FrenchPeterson1996, Gomez2025}, or when using maximal regularity splines in time \cite{FerrariFraschiniLoliPerugia2024}. Whether this is true in general remains an open question.

\section*{Funding}
\noindent
This research was funded in part by the Austrian Science Fund (FWF) \href{https://doi.org/10.55776/F65}{10.55776/F65}. MF and EZ are members of the Gruppo Nazionale Calcolo Scientifico-Istituto Nazionale di Alta Matematica (GNCS-INdAM).

\bibliography{bibliography}{}

\begin{thebibliography}{10}

\bibitem{AbdulleHenning2017}
A.~Abdulle and P.~Henning.
\newblock Localized orthogonal decomposition method for the wave equation with
  a continuum of scales.
\newblock {\em Math. Comp.}, 86(304):549--587, 2017.

\bibitem{AnselmannBause2020}
M.~Anselmann and M.~Bause.
\newblock Numerical {S}tudy of {G}alerkin--{C}ollocation {A}pproximation in
  {T}ime for the {W}ave {E}quation.
\newblock In {\em Mathematics of Wave Phenomena}, pages 15--36, Cham, 2020.
  Springer.

\bibitem{AnselmannBauseBecherMatthies2020}
M.~Anselmann, M.~Bause, S.~Becher, and G.~Matthies.
\newblock Galerkin-collocation approximation in time for the wave equation and
  its post-processing.
\newblock {\em ESAIM Math. Model. Numer. Anal.}, 54(6):2099--2123, 2020.

\bibitem{AzizMonk1989}
A.~K. Aziz and P.~Monk.
\newblock Continuous finite elements in space and time for the heat equation.
\newblock {\em Math. Comp.}, 52(186):255--274, 1989.

\bibitem{BalesLasiecka1994}
L.~Bales and I.~Lasiecka.
\newblock Continuous finite elements in space and time for the nonhomogeneous
  wave equation.
\newblock {\em Comput. Math. Appl.}, 27(3):91--102, 1994.

\bibitem{BignardiMoiola2023}
P.~Bignardi and A.~Moiola.
\newblock A space-time continuous and coercive formulation for the wave
  equation.
\newblock {\em \textit{arXiv:2312.07268}}, 2023.

\bibitem{DemkowiczGopalakrishnanNagarajSepulveda2017}
L.~Demkowicz, J.~Gopalakrishnan, S.~Nagaraj, and P.~Sep\'ulveda.
\newblock A spacetime {DPG} method for the {S}chr\"odinger equation.
\newblock {\em SIAM J. Numer. Anal.}, 55(4):1740--1759, 2017.

\bibitem{DongMascottoWang2024}
Z.~Dong, L.~Mascotto, and Z.~Wang.
\newblock A priori and a posteriori error estimates of a {D}{G}-{C}{G} method
  for the wave equation in second order formulation.
\newblock {\em arXiv:2411.03264}, 2024.

\bibitem{ErnGuermond2021a}
A.~Ern and J.-L. Guermond.
\newblock {\em Finite Elements I}, volume~72 of {\em Texts in Applied
  Mathematics}.
\newblock Springer Cham, 2021.

\bibitem{EvansBook}
L.~C. Evans.
\newblock {\em Partial differential equations}, volume~19 of {\em Graduate
  Studies in Mathematics}.
\newblock American Mathematical Society, Providence, RI, second edition, 2010.

\bibitem{FerrariCodes2025}
M.~Ferrari.
\newblock X{T}-{W}aves-{E}xp-{F}irst-{O}rder.
\newblock
  \url{https://github.com/MatteoFerrari11/XT-Waves-Exp-First-Order.git}, 2025.

\bibitem{FerrariFraschini2024}
M.~Ferrari and S.~Fraschini.
\newblock Stability of conforming space–time isogeometric methods for the
  wave equation.
\newblock {\em Math. Comp.}, 2025.
\newblock In press, doi: 10.1090/mcom/4062.

\bibitem{FerrariFraschiniLoliPerugia2024}
M.~Ferrari, S.~Fraschini, G.~Loli, and I.~Perugia.
\newblock Unconditionally stable space-time isogeometric discretization for the
  wave equation in {H}amiltonian formulation.
\newblock {\em arXiv:2411.00650}, 2024.

\bibitem{FerrariLoscherZank2025}
M.~Ferrari, R.~L{\"o}scher, and M.~Zank.
\newblock Stability of a modified {H}ilbert transform for second-order initial
  boundary value problems associated with space--time discretizations of the
  wave equation.
\newblock {\em in preparation}, 2025.

\bibitem{FerrariPerugia2025}
M.~Ferrari and I.~Perugia.
\newblock Space-time discretization of the wave equation in a
  second-order-in-time formulation: a conforming, unconditionally stable
  method.
\newblock {\em \textit{arXiv:2503.11166}}, 2025.

\bibitem{FraschiniLoliMoiolaSangalli2023}
S.~Fraschini, G.~Loli, A.~Moiola, and G.~Sangalli.
\newblock An unconditionally stable space–time isogeometric method for the
  acoustic wave equation.
\newblock {\em Comput. Math. Appl.}, 169:205--222, 2024.

\bibitem{French1993}
D.~A. French.
\newblock A space-time finite element method for the wave equation.
\newblock {\em Comput. Methods Appl. Mech. Engrg.}, 107(1-2):145--157, 1993.

\bibitem{FrenchPeterson1996}
D.~A. French and T.~E. Peterson.
\newblock A continuous space-time finite element method for the wave equation.
\newblock {\em Math. Comp.}, 65(214):491--506, 1996.

\bibitem{FuhrerGonzalezKarkulik2025}
T.~Führer, R.~González, and M.~Karkulik.
\newblock Well-posedness of first-order acoustic wave equations and space-time
  finite element approximation.
\newblock {\em IMA J. Numer. Anal.}, page drae104, 03 2025.

\bibitem{Gomez2025}
S.~Gómez.
\newblock A variational approach to the analysis of the continuous space-time
  {FEM} for the wave equation.
\newblock {\em \textit{arXiv:2501.11494}}, 2025.

\bibitem{HenningPalittaSimonciniUrban2022}
J.~Henning, D.~Palitta, V.~Simoncini, and K.~Urban.
\newblock An ultraweak space-time variational formulation for the wave
  equation: analysis and efficient numerical solution.
\newblock {\em ESAIM Math. Model. Numer. Anal.}, 56(4):1173--1198, 2022.

\bibitem{KotheLoscherSteinbach2023}
C. Köthe, R. Löscher, and O. Steinbach.
\newblock Adaptive least-squares space-time finite element methods.
\newblock {\em \textit{arXiv:2309.14300}}, 2023.

\bibitem{LionsMagenes1972}
J.-L. Lions and E.~Magenes.
\newblock {\em Non-homogeneous boundary value problems and applications. {V}ol.
  {I}}, volume Band 181 of {\em Die Grundlehren der mathematischen
  Wissenschaften}.
\newblock Springer-Verlag, New York-Heidelberg, 1972.
\newblock Translated from the French by P. Kenneth.

\bibitem{LoscherSteinbachZank2023}
R.~L{\"o}scher, O.~Steinbach, and M.~Zank.
\newblock Numerical results for an unconditionally stable space--time finite
  element method for the wave equation.
\newblock In {\em Domain Decomposition Methods in Science and Engineering
  XXVI}, pages 625--632. Springer, 2023.

\bibitem{LycheManniSpeleers2018}
T.~Lyche, C.~Manni, and H.~Speleers.
\newblock {\em Foundations of Spline Theory: B-Splines, Spline Approximation,
  and Hierarchical Refinement}, pages 1--76.
\newblock Springer International Publishing, Cham, 2018.

\bibitem{Steinbach2015}
O.~Steinbach.
\newblock Space-time finite element methods for parabolic problems.
\newblock {\em Comput. Methods Appl. Math.}, 15(4):551--566, 2015.

\bibitem{SteinbachZank2019}
O.~Steinbach and M.~Zank.
\newblock {\em A stabilized space--time finite element method for the wave
  equation}, volume 128 of {\em Lect. Notes Comput. Sci. Eng.}
\newblock Springer International Publishing, Cham, 2019.

\bibitem{SteinbachZank2020}
O.~Steinbach and M.~Zank.
\newblock Coercive space--time finite element methods for initial boundary
  value problems.
\newblock {\em Electron. Trans. Numer. Anal.}, 52:154--194, 2020.

\bibitem{Walkington2014}
N.~J. Walkington.
\newblock Combined {DG}-{CG} time stepping for wave equations.
\newblock {\em SIAM J. Numer. Anal.}, 52(3):1398--1417, 2014.

\bibitem{Zank2021}
M.~Zank.
\newblock {H}igher-{O}rder {S}pace--{T}ime {C}ontinuous {G}alerkin {M}ethods
  for the {W}ave {E}quation.
\newblock In {\em 14th WCCM-ECCOMAS Congress 2020}, volume 700, 2021.

\end{thebibliography}
\bibliographystyle{plain}

\end{document}